\documentclass[12pt]{amsart}

\usepackage{graphicx, verbatim,amsmath,amssymb, float}
\usepackage{hyperref}
\hypersetup{
    colorlinks=true,
    linkcolor=blue,
    filecolor=magenta,      
    urlcolor=cyan,
}

\newcommand{\E}{\mathbb{E}}
\newcommand{\R}{\mathbb{R}}

\newcommand{\1}{\mathbf{1}}
\newcommand{\calG}{\mathcal{G}}
\newcommand{\calM}{\mathcal{M}}
\newcommand{\calN}{\mathcal{N}}

\newcommand{\be}{\begin{equation}}
\newcommand{\ee}{\end{equation}}

\newcommand{\inr}[2]{\langle #1, #2 \rangle}
\DeclareMathOperator{\sgn}{sgn}
\DeclareMathOperator{\tr}{tr}

\DeclareMathOperator{\argmin}{argmin}

\DeclareMathOperator{\Var}{Var}
\DeclareMathOperator{\Cov}{Cov}

\newtheorem{conjecture}{Conjecture}
\newtheorem{theorem}[conjecture]{Theorem}
\newtheorem{lemma}[conjecture]{Lemma}
\newtheorem{corollary}[conjecture]{Corollary}
\newtheorem{definition}[conjecture]{Definition}
\newtheorem{proposition}[conjecture]{Proposition}
\newtheorem{remark}[conjecture]{Remark}
\newcommand{\pp}{q}

\title{Moderate Deviations in Cycle Count}
\date{\today}

\author{
Joe Neeman
\and Charles Radin
\and Lorenzo Sadun
}
\address{Joe Neeman\\Department of Mathematics\\The University of
  Texas at Austin\\ Austin, TX 78712} \email{joeneeman@gmail.com}
\address{Charles Radin\\Department of Mathematics\\The University of
  Texas at Austin\\ Austin, TX 78712} \email{radin@math.utexas.edu}
\address{Lorenzo Sadun\\Department of Mathematics\\The University of
  Texas at Austin\\ Austin, TX 78712} \email{sadun@math.utexas.edu}

\thanks{
This work was partially supported by the Deutsche Forschungsgemeinschaft (DFG, German Research
Foundation) under Germany's Excellence Strategy – EXC-2047/1 – 390685813,
and by a fellowship from the Alfred P. Sloan Foundation.
}

\begin{document}

\begin{abstract}
  We prove moderate deviations bounds for the lower tail of the number of odd cycles in a
  $\calG(n, m)$ random graph. We show that the probability of decreasing triangle density by $t^3$,
  is $\exp(-\Theta(n^2 t^2))$ whenever $n^{-3/4} \ll t^3 \ll 1$, while for $k \ge 5$ 
  we give the same estimate for the probability of decreasing the $k$-cycle density by $t^k$, but for
  the larger range $n^{-1} \ll t^k \ll 1$.
  When $m \ge \frac 12 \binom n2$, we also find the leading coefficient in the exponent.
  This complements results of 
  Goldschmidt et al., who showed that for $n^{-3/2} \ll t^k \ll n^{-1}$, the probability is 
  $\exp(-\Theta(n^3 t^{2k}))$. That is, deviations of order
  smaller than $n^{-1}$ behave like small deviations, and deviations of order larger than $n^{-3/4}$ (for triangles)
  or $n^{-1}$ (for $k$-cycles with $k \ge 5$) behave like large deviations. For triangles, we conjecture that a sharp change
  between the two regimes occurs for deviations of size $n^{-3/4}$,
  which we associate with  
  a single large
  negative eigenvalue of the adjacency matrix becoming responsible for almost all of the cycle
  deficit. 

Our results can be interpreted as finite size effects in phase transitions in
    constrained random graphs.


\end{abstract}

\maketitle

\section{Introduction}

We prove moderate deviations bounds for the lower tail of the number of odd $k$-cycles in a
$\calG(n, m)$ random graph, i.e.\ a uniformly random graph among
all the graphs with $n$ vertices and $m$ edges.
We study deviations larger than those
of
Goldschmidt et al.~\cite{GoldschmidtGriffithsScott20} but smaller
than large deviations, which are of order the mean of the cycle density. For
instance, with the notation that $\tau_3(G)$ is the triangle density of a $\calG(n, m)$ graph $G$
where $n \to \infty$ and $m = p \binom{n}{2} + O(1)$, for some $1/2
\le p <1$ that is
fixed as $n \to \infty$ and $n^{-3/4} \ll t^3 \ll 1$, we prove (see Theorem~\ref{thm:large-p-cycle-count})
that
    \be
        \Pr\left(\tau_3(G) \le p^3 - t^3\right) =
        \exp\left(
            - \frac{\ln \frac{p}{1-p}}{2(2p - 1)} t^{2} n^2 + o(t^{2} n^2)
        \right).
    \ee

The number of triangles in a random graph is a fundamental and surprisingly important
random variable in the study of probabilistic combinatorics.
The probabilistic behavior of these triangle counts is at least partially
responsible for the development of many important
methods related to concentration inequalities for dependent random variables,
including Janson's inequality~\cite{Janson90}, the entropy
method~\cite{BoucheronLugosiMassart03}, martingale difference techniques in
random graphs, and others~\cite{Chatterjee12}.

The traditional point of view, as exemplified by the seminal paper by Janson
and Ruc\'inski~\cite{JansonRucinski02}, holds that the lower tail of the triangle count is easy to
characterize while the upper tail is hard.
This view stems at least partly
from the fact that most earlier works studied the $\calG(n, p)$ model for $p \to 0$,
and a substantial part of the difficulty in the study of the upper tail is to obtain
the correct dependence on $p$.
For dense graphs (i.e.\ when $p$ is fixed), the lower tail has more subtle behavior,
as was noted already by~\cite{Zhao17}.
In this regime the
$\calG(n, m)$ model, in which the number of edges is fixed at $m$, 
differs substantially from the $\calG(n,p)$ model.
For example, one can easily see that under
$\calG(n, p)$, the number of triangles, $T_3(G)$, satisfies
$\Var(T_3(G)) = \Theta(n^4)$, while under $\calG(n, m)$,
$\Var(T_3(G)) = \Theta(n^3)$. The distinction between the two models
-- especially in the lower tail --
becomes even more pronounced at larger deviations. This can be intuitively
explained by the fact that in $\calG(n, p)$ one can easily ``depress'' the
triangle count simply by reducing the number of edges: a graph $G$
with edge number $|E(G)|
\approx q \binom{n}{2}$ will typically have triangle density $\tau_3 \approx q^3$, and the
probability of seeing such a graph under $\calG(n, p)$ is of the order
$\exp(-\Theta(n^2 (p-q)^2))$; it follows that under $\calG(n, p)$ we have
\be
    \Pr(\tau_3(G) \le \E \tau_3(G) - t^3) \ge \exp(-\Omega(n^2 t^6)).
\ee

Under $\calG(n, m)$, large deficits in the triangle density are much rarer than they
are in $\calG(n,p)$. At the scale of constant-order deficits, this was noticed
in~\cite{RadinSadun13, RadinSadun15},
where it is proved that for $t = \Theta(1)$ and $\calG(n,m)$ with $m = \Theta(n^2)$,
\be
    \Pr(\tau_3(G) \le \E \tau_3(G) - t^3) = \exp(-\Theta(n^2 t^{2})).
\ee
(They also found the exact leading-order term in the exponent when $m = \frac 12 \binom{n}{2}+ o(n^2)$
and bounded the leading-order coefficient for all other values of $m$.)
The same argument also works for odd $k > 3$.
At the other end of the scale, a recent result of Goldschmidt et al.~\cite{GoldschmidtGriffithsScott20}
showed that for $n^{-3/2} \ll t^k \ll n^{-1}$ the lower tail has a different behavior:
\be
    \Pr(\tau_k(G) \le \E \tau_k(G) - t^k) = \exp(-\Theta(n^3 t^{2k})).
\ee
(Again, they also found the exact leading-order term in the exponent.)
Since $t^k \le \Theta(n^{-3/2})$ is within the range of the Central Limit Theorem this leaves
open the case of $n^{-1} \ll t^k \ll 1$. Noting that the two exponential rates (namely $n^2 t^{2}$ and $n^3 t^{2k}$)
cross over at $t^k =\Theta(n^{-\frac{k}{2(k-1)}})$, it is natural to guess that for all odd $k$,
\be
    \label{eq:two-regimes}
    \Pr(\tau_k(G) \le \E \tau_k(G) - t^k) = \begin{cases}
        \exp(-\Theta(n^3 t^{2k})) &\text{if $t^k \ll n^{-\frac{k}{2(k-1)}}$}, \\
        \exp(-\Theta(n^2 t^{2})) &\text{if $n^{-\frac{k}{2(k-1)}} \ll t^k \ll 1$}.
    \end{cases}
\ee
In the case $k=3$, we prove the second of these two cases; the first remains a conjecture.
For $k \ge 5$,~\eqref{eq:two-regimes} turns out to be false: the 
boundary between the two regimes turns out to occur when $t^k$ is of the order $n^{-1}$.
This is perhaps surprising because it implies that a deviation of order $n^{-1 - \epsilon}$
has probability $\exp(-\Theta(n^{1 + 2\epsilon}))$ but a deviation of order $n^{-1 + \epsilon}$ has
the much smaller probability $\exp(-n^{2 - 2/k - O(\epsilon)})$.

We also prove some structural results on
graphs with $\tau_k(G) \le \E \tau_k(G) - t^k$ in our range of $t^k$: conditioned on this cycle-count deviation,
with high probability such a graph has a very negative eigenvalue, and also has a small subgraph with substantially
smaller edge density. These structural results provide a plausible
explanation for the importance of the threshold between the two regimes: it is the threshold at which a single
large negative eigenvalue of the adjacency matrix becomes responsible for almost all of the $k$-cycle
deficit.

\section{Context and references}

We are concerned with random graphs $\calG(n, m)$, the uniform distribution on graphs on $n$ nodes with $m$ edges.
For a graph $G$ and an integer $k \ge 3$, define
$T_k(G)$ to be the number of injective maps $\phi: \{1, \dots, k\} \to V(G)$
for which $\{\phi(1), \phi(2)\}, \{\phi(2), \phi(3)\}, \dots, \{\phi(k), \phi(1)\}$ are all edges of $G$;
we say that $T_k(G)$ is the number of $k$-cycles in $G$.
The \emph{$k$-cycle density} is $\tau_k(G) = \frac{1}{\binom nk} T_k(G) \in [0, 1]$.
Results
on the probability of deviations of subgraph density from the mean
fall into three classes by size: small deviations, on the order of the
standard deviation, large deviations, on the order of the mean, and
moderate deviations, of intermediate size. 

Our main results concern the moderate regime of deviations of cycle
density in $\calG(n, m)$, in which we prove, among other things, that
deviations near but below the large class are qualitatively different
from deviations near but above the small class. We know of no other
results of this sort, for $\calG(n, m)$ or the $\calG(n, p)$ random graph model, in which edges appear independently.

For small deviations there is a long history under the name Central
Limit Theorem. There are also many papers on moderate and large
deviations of subgraph counts. As background, more specifically for
results discussed here, we suggest the following: 
\cite{ChatterjeeDembo20,CookDembo20,LubetzkyZhao17,Bhattacharya_etal17,Zhao17,
BhattacharyaDembo19,Gunby20,LubetzkyZhao15,JansonWarnke16} and references
within them for a broader view. As our results are strongly colored
by large deviations we note in particular \cite{Chatterjee16}.

For convenience we note some common asymptotics notation. We use $f
=o(g)$ or $f \ll g$ to mean $\lim |f(n)|/g(n)=0$, $f = O(g)$ to mean $\lim\sup f(n)/g(n) < \infty$,
$f=\Omega(g)$ to mean $\lim\inf
f(n)/g(n) >0$, $f=\omega(g)$ or $f \gg g$ to mean $\lim |f|/g = \infty$, and
$f=\Theta(g)$ to mean both $f = O(g)$ and $f = \Omega(g)$.
The phrase ``with high probability'' means ``with probability converging to 1 as $n \to \infty$,''
and we also make use of probabilistic asymptotic notation: ``$f = O(g)$ with high probability'' means
that for every $\epsilon > 0$ there exists $C > 0$ with $\limsup \Pr(f \ge C g) \le \epsilon$;
``$f = o(g)$ with high probability'' means that for every $\epsilon > 0$, $|f|/g \le \epsilon$ with high 
probability; and analogously for $\Omega$ and $\omega$.

We are studying the $k$-cycle density of $\calG(n,m)$ for $t \to 0$ (but not too quickly)
and for odd $k$ (for even $k$, it is not possible for $\tau_k(G)$ to be significantly smaller than $p^k$).
The case $0 \le t^k \le \Omega(n^{-3/2})$ is within the range of the
Central Limit Theorem and it is covered by Janson's more general work
on subgraph statistics~\cite{Janson94}. The range $n^{- 3/2} \ll t^k \ll
n^{-1}$ is studied by~\cite{GoldschmidtGriffithsScott20}; they showed
that in this regime
\be\label{eq:ggs-bound-large-deviations}
  \Pr(\tau_k(G) \le \E \tau_k(G) - t^k) = \exp\left(-\frac{t^{2k} n^3}{2\sigma_p^2}(1 + o(1)) \right),
\ee
where
$\sigma_p^2 = \Var(\tau_k(G)) / n^3$, which is of constant order.  They
also show an upper bound for larger $t$: for $n^{-1} \ll t^k \ll 1$,
\be \Pr(\tau_k(G) \le \E \tau_k(G) - t^k) = \exp\left( -\Omega(t^{k} n^2) \right).
\ee We show that this upper bound is mostly not tight. In particular, for $k$-cycles
with $k \ge 5$ we show that the correct exponent is $t^2 n^2$ for all $n^{-1} \ll t^k \ll 1$.
For triangles, we show the same exponent but only in the range $n^{-3/4} \ll t^3 \ll 1$;
we conjecture that this is the best possible range, and that
the bound~\eqref{eq:ggs-bound-large-deviations} is sharp for triangles in the range $n^{-1} \ll t^3 \ll n^{-3/4}$.
In the case $p \ge \frac
12$, we also derive more detailed results (see Theorem~\ref{thm:large-p-cycle-count}): we identify the leading
constant in the exponent and we prove some results on the graph structure conditioned
on having few cycles.

\subsection{Related work on random graphs}

Besides the work of~\cite{GoldschmidtGriffithsScott20}, there is
related work on large deviation principles (LDPs) for more general statistics, and LDPs for sparser
graphs, notably in \cite{HarelMoussetsamotij19,CookDembo20}.
In particular,~\cite{HarelMoussetsamotij19} is the only existing work we know of in which
the conditional structure of subgraph-density-constrained random graphs is established.
Specifically, they show that for sparse random graphs conditioned on having more than the
expected number of cliques, the random graph has either a ``clique'' structure in which
there is a collection of vertices has higher-than-expected edge density or a ``hub'' structure
in which there is a partition of the vertices with a higher-than-expected edge
density between the two parts. In contrast, our results show that for dense random graphs
with fewer cycles than expected, there is a collection of vertices with lower-than-expected
edge density.

\bigskip

Moderate deviations in triangle count (i.e. the case $k=3$) in $\calG(n,m)$ can be seen from
a different vantage based on~\cite{NeemanRadinSadun20}. That paper
follows a series of 
works~\cite{RadinSadun13,RadinSadun15,RRS2,KRRS1,KRRS2,KRRS3,Ko}
on the asymptotics of `constrained'
random graphs, in particular the asymptotics of $\calG(n,m,t)$, the
uniform distribution on graphs on $n$ nodes constrained to have $m$
edges and $t$ triangles. A large deviation principle, using
optimization over graphons, a variant of the seminal
work~\cite{ChatterjeeVaradhan11} by Chatterjee and Varadhan on large
deviations in $\calG(n,p)$, was used to prove various features of
phase transitions between asymptotic `phases', phases illustrated by
the entropy-optimal graphons. (See also~\cite{DemboLubetzky18}.) But
in~\cite{NeemanRadinSadun20} numerical evidence showed that the
transitions could be clearly seen in finite systems, using constrained
graphs with as few as 30 vertices. From this perspective moderate
deviations in triangle count can be understood as {\em finite size
  effects in a phase transition}. Asymptotically, entropy goes through
a sharp ridge as the edge density/triangle density pair
$(\varepsilon,\tau)$ passes through $(\varepsilon,\varepsilon^{3})$
(Thms.\,1.1,1.2 in~\cite{RadinSadun15}), and moderate deviations
quantify how the sharp ridge rounds off at finite node number,
somewhat as an ice cube freezing in water has rounded edges. The focus
thus shifts to the infinite system, where emergent phases are
meaningful, away from $\calG(n,m,t)$ or $\calG(n,m)$.

\subsection{Related work on random matrices}

Since we are studying the spectrum of the adjacency matrix, our methods mainly
come from random matrix theory. Specifically, we are interested in large deviations
of eigenvalues of the random adjacency matrices coming from our random graphs.
The study of large deviations of eigenvalues is an active topic, but the results
we aim for are somewhat atypical. Traditionally, ``large deviations'' refers to
deviations on the order of the mean, so large deviations results for random
matrices typically consider the event that the largest eigenvalue of
a symmetric $n \times n$ matrix with i.i.d. mean-zero, variance-$\sigma^2$ entries
is of order $\alpha \sqrt n$ for $\alpha > 2 \sigma$; this is because the typical
value of the largest eigenvalue is of order $2 \sigma \sqrt n$.
However, because an eigenvalue of order $n^\beta$ contributes $n^{k\beta}$ to the $k$-cycle count,
and because we are interested in cycle-count deviation of orders larger than $n^{k/2}$,
we are necessarily interested in much larger eigenvalues.

Another difference in our work is that we consider several large eigenvalues
simultaneously. This is because we need to consider the possibility
that the cycle count is affected by several atypically large eigenvalues instead
of just one.

In related works,
\begin{itemize}
    \item Guionnet and Husson~\cite{GuionnetHusson20} showed an LDP
        for the largest eigenvalue for a family of random matrices that
        includes Rademacher matrices,
        which is essentially the case that we consider when $p=\frac 12$.
    \item Augeri~\cite{Augeri16} showed an LDP for the largest eigenvalue
        for random matrices whose entries have heavier-than-Gaussian tails.
    \item Battacharya and Ganguly~\cite{BhattacharyaGanguly20} showed an LDP
        for the largest two eigenvalues of a sparse Erd\H{o}s-R\'enyi graph.
        The methods we use for our eigenvalue LDPs are related
        to their
        methods for the second-largest eigenvalue. In order to make the connection
        to cycle counts, however, we need to handle the entire spectrum.
    \item Augeri, Guionnet, and Husson~\cite{AugeriGuionnetHusson20} showed an LDP
        for the largest eigenvalue for most random matrices with subgaussian
        elements. These are essentially the same random matrices that we consider,
        with the main difference being that they are looking at eigenvalues
        of size $\Theta(\sqrt n)$.
\end{itemize}

\section{Cycle counts}

Our general setting is: we let $A$ be the adjacency matrix of a $\calG(n, m)$ graph,
where $n \to \infty$ and $m = p \binom{n}{2} + O(1)$, for some $p \in \R$ that is
fixed as $n \to \infty$. We denote by $\tau_k(A)$ the $k$-cycle density
of $A$, and we order the eigenvalues $\lambda_1(A) \ge \cdots \ge \lambda_n(A)$ in non-increasing order.

We prove two theorems governing asymptotic behavior as $n \to \infty$. We define the critical exponent
\be\label{eq:c*-def}
    c_* = \min\{1, \frac{k(2-k)}{2k-2}\}
\ee
and we assume that $n^{-c_*} \ll t^k \ll 1$; this is equivalent to $n^{-3/4} \ll t^3 \ll 1$ for $k=3$, and $n^{-1} \ll t^k \ll 1$
for $k \ge 5$.
Our first theorem is a strong result for $\frac 12 \le p < 1$.
\begin{theorem}\label{thm:large-p-cycle-count}
    If $\frac 12 \le p < 1$ and $n^{-c_*} \ll t^k \ll 1$ then
    \be
        \Pr\left(\tau_k(A) \le p^k - t^k\right) =
        \exp\left(
            - \frac{\ln \frac{1-p}{p}}{2(1-2p)} t^{2} n^2 + o(t^{2} n^2)
        \right),
    \ee
    with the convention that $\frac{\ln \frac{1-p}{p}}{1-2p} = 2$ when $p = \frac 12$.
    Moreover, conditioned on $\tau_k(A) \le p^k - t^k$, with high probability we have
    \be
        \lambda_n(A) = -t n (1 - o(1))
    \ee
    and $\lambda_{n-1}(A) \ge - o(t n)$.
\end{theorem}

The second result, for $0 < p \le \frac12$, is  weaker.

\begin{theorem}\label{thm:small-p-cycle-count}
    If $0 < p \le \frac 12$ and $n^{-c_*} \ll t^k \ll 1$ then $\Pr\left(\tau_k(G) \le p^k - t^k\right)$
    is bounded above by 
    \be
        \exp\left(
            - \frac{\ln \frac{p}{1-p}}{2(2p - 1)} t^{2} n^2 + o(t^{2} n^2)
        \right) 
    \ee
    and bounded below by 
    \be \exp\left(
            -\frac{1}{2p(1-p)} t^{2} n^2 + o(t^{2} n^2)
        \right).
        \ee
    Moreover, conditioned on $\tau_k(A) \le p^k - t^k$, with high probability we have
    \be
        \lambda_n(A) = -\Omega(t n).
    \ee
\end{theorem}
Together, these theorems show that $\Pr(\tau_k(A) \le p^k - t^k) = \exp (-\Theta(t^{2}n^2))$ for 
all $0 < p < 1$ and $n^{-c_*} \ll t^k \ll 1$.

In the case $p \ge \frac 12$, we also give a graph-theoretic characterization
of the conditioned graph: given that $\tau_k(A) \le p^k - t^k$, the graph contains
a lower-density subgraph of about $t n / (2p-1)$ vertices.
In what follows, for $V_1, V_2 \subset V$, let
\[
    E(V_1, V_2) = \sum_{v_1 \in V_1, v_2 \in V_2} 1_{\{\{v_1, v_2\} \in E(G)\}}
\]
count the edges between $V_1$ and $V_2$, while double-counting those edges with both endpoints in
$V_1 \cap V_2$.

\begin{theorem}\label{thm:conditional-structure}
    If $\frac 12 \le p < 1$ and $n^{-k/(2(k-1))} \ll t \ll 1$ then conditioned on $\tau_k(G) \le p^k - t^k$,
    with high probability there exists a subset $U \subset V(G)$ of size $|U| = (1 + o(1)) t n / (2p-1)$
    such that for every $V_1, V_2 \subset V(G)$,
    \[
        E(V_1, V_2) = p |V_1| |V_2| - (2p-1) |V_1 \cap U| |V_2 \cap U| + o(t n |V_1 \cup V_2|).
    \]
\end{theorem}

In particular, setting $V_1 = V_2 = U$ shows that the subgraph induced by $U$ has edge density about $1 - p$.
More generally, Theorem~\ref{thm:conditional-structure} implies that $G$ has no other non-trivial structure
at the scale of $t n$ or more vertices.

\subsection{Centering the matrix}
The main point of this section is that when considering the lower tail
for cycle counts in $\calG(n, m)$ graphs, it suffices to look at
eigenvalues of the centered adjacency matrix.
This might sound obvious, but there are two subtleties:
\begin{enumerate}
    \item It is important that we are looking at the lower tail, because the
        upper tail probabilities are controlled by perturbations to the largest
        eigenvector; this is exactly the eigenvector that gets destroyed when we
        center the adjacency matrix, so the eigenvalues of the centered adjacency
        matrix don't give much information about the upper tail probabilities.
    \item It is important that we are looking at $\calG(n, m)$ and not $\calG(n, p)$,
        because -- as discussed in the introduction -- in $\calG(n, p)$ the entropically favorable way to reduce the $k$-cycle count
        is to reduce the number of edges; again, this primarily affects the largest
        eigenvector and so is not related to the centered adjacency matrix.
\end{enumerate}

\begin{lemma}\label{lem:centered-cycle-count}
    Let $A$ be the adjacency matrix of a graph with $n$ vertices and $m = p \binom n2$ edges,
    and let $d_i$ be the degree of vertex $i$.
    Let $\tilde A = A - p \1 + pI$. For any $k \ge 3$, there exists $\epsilon > 0$ such that
    if $\|\tilde A\|_{\text{op}} \le \epsilon n$ then
    \be
        \tr[\tilde A^k] = \tr[A^k] - p^k n^k - (1 - O(\epsilon)) k n^{k-3} \sum_i (d_i - pn)^2 + O(n^{k-1})
    \ee
\end{lemma}

\begin{proof}
    Let $B = \tilde A + p \1$ and consider $\tr[B^k]$. (The extra contribution of $p I$ in $A$ makes a lower-order
    contribution and we will handle it later.)
    Consider the various terms in the expansion $(\tilde A + p \1)^k$ according to how many copies of $\tilde A$
    they contain: there is a $\tilde A^k$ term and a $p^k n^{k-1} \1$ term (which has trace $p^k n^k$),
    and every other term is a product involving at least one occurence of $\1$
    and at least one occurrence of $\tilde A$. Note that $\1 \tilde A \1 = 0$,
    and so all the terms that have exactly one occurrence of $\tilde A$ vanish;
    and of the terms containing exactly two occurrences of $\tilde A$, the only non-vanishing ones are
    of the form $\tilde A^2 \1_{k-2}$ (up to cyclic permutation). There are $k$ of these terms,
    and so after taking the trace, they contribute
    \begin{equation}\label{eq:contribution-of-squares}
        k \tr [\tilde A^2 \1^{k-2}] =  k n^{k-3} \tr[\tilde A^2 \1] = k n^{k-3} |\tilde A 1|^2 = k n^{k-3} \sum_i (d_i - pn)^2
    \end{equation}
    to $\tr[B^k]$.

    Next, consider the terms containing more than two occurrences of $\tilde A$. Since $\1 \tilde A \1 = 0$,
    the only non-vanishing contributions take the form
    \[
        \tr \prod_{i=1}^m \1^{j_i} \tilde A^{\ell_i}
    \]
    for some $\ell_i \ge 2$, and $\sum_i \ell_i \ge 3$. Since $\1^j = n^{j-1} \1$, the term displayed above can be re-written
    (setting $j = \sum j_i = k - \sum \ell_i$) as
    \[
        n^{j-m} \tr \prod_{i=1}^m \1 \tilde A^{\ell_i} = n^{j-m} \prod_{i=1}^m 1^T \tilde A^{\ell_i} 1.
    \]
    Since each $\ell_i \ge 2$, if $\|\tilde A\|_{\text{op}} \le \epsilon n$ then
    $|1^T \tilde A^{\ell_i} 1| \le \|\tilde A^{\ell_i - 2}\|_{\text{op}} |\tilde A 1|^2 \le \epsilon^{\ell_i - 2} n^{\ell_i - 2} |\tilde A 1|^2 \le \epsilon^{\ell_i} n^{\ell_i + 1}$.
    Now we consider two cases: if $\ell_i = 2$ for all $i$ then $m \ge 2$ (because $\sum_i \ell_i \ge 3$). In this case,
    we use the bound $|1^T \tilde A^{\ell_i} 1| \le \epsilon^{\ell_i} n^{\ell_i + 1} \le \epsilon^2 n^{\ell_i + 1}$ for $i \ge 2$
    and the bound
    $|1^T \tilde A^{\ell_i} 1| \le \epsilon^{\ell_i - 2} n^{\ell_i -2} |\tilde A 1|^2 \le n^{\ell_i - 2} |\tilde A 1|^2$ for $i = 1$, to obtain
    \begin{equation}\label{eq:contribution-of-2-squares}
        \left|\tr \prod_{i=1}^m \1^{j_i} \tilde A^{\ell_i}\right|
        \le \epsilon^2 n^{k - 3} |\tilde A 1|^2.
    \end{equation}
    (The $k-3$ exponent on $n$ comes from the fact that $j - m + \ell_1 - 2 + \sum_{i = 2}^m (\ell_i + 1) = j + \sum_i \ell_i - 3 = k - 3$.)
    On the other hand, if there is some $i$ with $\ell_i \ge 3$ then without loss of generality $i = 1$; we apply the bound
    $|1^T \tilde A^{\ell_i} 1| \le \epsilon^{\ell_i} n^{\ell_i + 1} \le n^{\ell_i + 1}$ for $i \ge 2$
    and the bound
    $|1^T \tilde A^{\ell_i} 1| \le \epsilon^{\ell_i - 2} n^{\ell_i -2} |\tilde A 1|^2 \le \epsilon n^{\ell_i - 2} |\tilde A 1|^2$ for $i = 1$, to obtain
    \begin{equation}\label{eq:contribution-of-cubes}
        \left|\tr \prod_{i=1}^m \1^{j_i} \tilde A^{\ell_i}\right|
        \le \epsilon n^{k - 3} |\tilde A 1|^2.
    \end{equation}
    Now compare~\eqref{eq:contribution-of-squares} to~\eqref{eq:contribution-of-2-squares} and~\eqref{eq:contribution-of-cubes}:
    out of all the terms in the expansion of $\tr[(\tilde A + p \1)^k]$ that contain between $1$ and $k-1$ copies of $\tilde A$,
    the terms containing two adjacent copies of $\tilde A$ (i.e.\ the terms we compute in~\eqref{eq:contribution-of-squares}
    dominate). Since the total number of terms in the expansion is $2^k$, we see that if $\epsilon$ is sufficiently small
    in terms of $k$ then
    \[
        \tr [B^k] = \tr [\tilde A^k] + p^k n^k + (1 - O(\epsilon)) k n^{k-3} \sum_i (d_i - pn)^2.
    \]

    Finally, to get the claim in terms of $A = B - p I$, note that $A^k = \sum_{j=0}^k \binom{k}{j} B^j (-p)^j$.
    We apply our previous result to each $B^j$ term, noting that for $j \ge 1$ each term contributes only $O(n^{k-1})$.
\end{proof}

Combining Lemma~\ref{lem:centered-cycle-count} with the observation that $\E \tr[A^k] = p^k n^k + O(n^{k-1})$ when
$A$ is the adjacency matrix of a $\calG(n, m)$ graph,
we arrive at the following consequence:

\begin{corollary}\label{cor:centered-cycle-count}
    Let $A$ be the adjacency matrix of a $\calG(n, m)$ graph and let $\tilde A = A - \E A$.
    For any $t \ge 0$ and all sufficiently small $\epsilon > 0$ depending on $k$,
    \be
        \Pr(\tr[A^k] \le \E \tr[A^k] - t^k)
        \le \Pr(\tr[\tilde A^k] \le -t^k + O(n^{k-1})) + \Pr(\|\tilde A\|_{\text{op}} \ge \epsilon n)
    \ee
\end{corollary}

\section{Large deviations for eigenvalues of random matrices}

In this section and beyond, we let $A$ denote a generic random matrix and we estimate
the most positive eigenvalues of $A$. Since we are looking at lower tails, the
most important such matrix to keep in mind is {\em minus} the centered adjacency matrix,
previously denoted $\tilde A$ or $A - \E A$. This is the same as {\em plus} the centered
adjacency matrix of a random graph with edge density $\pp=1-p$. The proof of Theorem 
\ref{thm:large-p-cycle-count} ($p \ge \frac12$) thus relies on results for $\pp \le \frac12$,
while the proof of Theorem \ref{thm:small-p-cycle-count} ($p \le \frac12$) relies on results for $\pp \ge \frac12$.

    \begin{definition}
        For a random variable $\xi$, its cumulant-generating function is
        \be
            \Lambda_\xi(s) = \ln \E \exp(s \xi)
        \ee
        whenever the expectation exists; when the expectation does not exist, we set $\Lambda_\xi(s) = +\infty$.
    \end{definition}

    \begin{definition}
        The random variable $\xi$ is \emph{subgaussian} if there exists a constant $C$
        such that $\Lambda_\xi(t) \le C t^2$ for every $t \in \R$.
    \end{definition}

    Note that according to our definition, a subgaussian random variable has mean zero
    (since if $\Lambda_\xi(t)$ is finite on a neighborhood of 0 then $\Lambda_\xi(0) = 0$ and
    $\Lambda_\xi'(0) = \E \xi$,
    and so if $\E \xi$ is non-zero then one cannot have $\Lambda_\xi(t) \le C t^2$ on a neighborhood of 0).
    Note also that if $\E \xi = 0$ and $\|\xi\|_\infty < \infty$ then $\xi$ is subgaussian.

    \begin{definition}
        For a function $f: \R \to \R$, its Legendre transform is the function $f^*: \R \to \R \cup \{+\infty\}$
        defined by
        \be
            f^*(y) = \sup_{x \in \R} \{xy - f(x)\}
        \ee
    \end{definition}

    Some basic properties of the Legendre transform include:
    \begin{itemize}
        \item If $f \le g$ then $f^* \ge g^*$.
        \item If $f$ is convex then $f^{**} = f$.
        \item If $f(x) =  c x^2$ then $f^*(x) = \frac{x^2}{4c}$.
    \end{itemize}

    Our goal in this note is to establish large deviations principles for
    extreme eigenvalues and singular values of random matrices.
    We will consider a symmetric $n \times n$ random matrix $A_n$ (or sometimes just $A$)
    having i.i.d.\
    upper-diagonal entries and zero diagonal entries. The letter $\xi$ will
    always denote a random variable that is distributed as an upper-diagonal
    element of $A$, and we will always assume that $\xi$ is subgaussian.
    We write $\lambda_i(A)$ for the eigenvalues of $A$ (in non-increasing order)
    and $\sigma_i(A)$ for the singular values of $A$ (in non-increasing order).

    For the definition of a large deviations principle (LDP), we refer
    to~\cite[Chapter~27]{Kallenberg01}.

    \begin{theorem}\label{thm:eigenvalue-ldp}
        Let $\xi$ be a subgaussian random variable.
        For any integer $k \ge 1$ and any sequence $m_n$ satisfying $\sqrt n \ll m_n \ll n$,
        the sequence
        \be
            \frac{1}{m_n} (\sigma_1(A_n), \dots, \sigma_k(A_n))
        \ee
        satisfies an LDP with speed $m_n^2$ and
        good rate function $I: \R_+^k \to [0, \infty)$ given by
        \be
            I(x) = \frac{|x|^2}{2} \inf_{s \in \R} \frac{\Lambda^*_\xi(s)}{s^2}.
        \ee
        If we assume in addition that the function
        $s \mapsto \frac{\Lambda^*_\xi(s)}{s^2}$ achieves its infimum at some $s \ge 0$,
        then the sequence
        \be
            \frac{1}{m_n} (\lambda_1(A_n), \dots, \lambda_k(A_n))
        \ee
        satisfies an LDP with speed $m_n^2$ and the same good rate function $I$ as above.
    \end{theorem}

    If $A_n$ is the centered adjacency matrix of $\calG(n, \pp)$ then it is covered by
    Theorem~\ref{thm:eigenvalue-ldp}, where $\xi$ is the random variable taking the values $-\pp$ and $1-\pp$
    with probabilities $1-\pp$ and $\pp$ respectively.
        In this case, we have
    \be
        \Lambda_\xi^*(s) = D(\pp + s, \pp) := (\pp + s) \ln \frac{\pp+s}{\pp} + (1 - \pp - s) \ln \frac{1-\pp-s}{1-\pp},
    \ee
    with the understanding that $\Lambda_\xi^*(s) = +\infty$ whenever $\pp + s \not \in (0, 1)$.
    It is not hard to check -- and we will do it in Section~\ref{sec:Gmn} --
    that $\frac{\Lambda_\xi^*(s)}{s^2}$ achieves its infimum at some $s \ge 0$
    if and only if $\pp \le \frac 12$.

    In the case that $\frac{\Lambda_\xi^*(s)}{s^2}$ saturates its infimum only at negative $s$
    (corresponding to $\pp > \frac 12$ in the Bernoulli example),
    we are not able to show an LDP for the eigenvalues. Note, however, that
    $\sum_i \sigma_i^2(A) \ge \sum_i \lambda_i^2(A)$ and so our LDP for singular values
    provides an upper bound: it implies, for example, that
    \be \label{eq:all-s}
        \frac{1}{m_n^2} \ln \Pr\left(\sqrt{\sum_i \lambda_i^2(A_n)} > m_n t\right) \le - \frac{t^2}{2} \inf_{s \in \R} \frac{\Lambda_\xi^*(s)}{s^2} + o(1)
    \ee
    On the other hand, we can also easily show the lower bound
    \be \label{eq:positive-s}
        \frac{1}{m_n^2} \ln \Pr\left(\sqrt{\sum_i \lambda_i^2(A_n)} > m_n t\right) \ge - \frac{t^2}{2} \inf_{s \ge 0} \frac{\Lambda_\xi^*(s)}{s^2} - o(1),
    \ee
    but the assumption that $\frac{\Lambda_\xi^*(s)}{s^2}$ saturates its infimum only at negative $s$
    implies that these bounds are non-matching.

    There are natural examples (including the Bernoulli example mentioned above)
    where $s^{-2} \Lambda_\xi^*(s)$ is increasing for $s \ge 0$.
    In this case,
    \be
        \inf_{s \ge 0} s^{-2} \Lambda_\xi^*(s) = \lim_{s \to 0} s^{-2} \Lambda_\xi^*(s) = \frac 12 (\Lambda_\xi^*)''(0) = \frac{1}{2 \E \xi^2},
    \ee
    and so our lower bound (for simplicity, focusing only on the case $k=1$) becomes
    \be
        \frac{1}{m_n^2} \ln \Pr\left(\lambda_1(A_n) > m_n t\right) \ge - \frac{t^2}{4 \E \xi^2} - o(1).
    \ee
    When $\xi$ has a Gaussian
    distribution, this turns out to be sharp, but we show that it is not sharp in general.

    \begin{theorem}\label{thm:lower-bound-non-sharp}
        In the setting of Theorem~\ref{thm:eigenvalue-ldp}, if $\E \xi^3 < 0$ and $\lim_{s \to \infty} s^{-2} \Lambda_\xi(s) = 0$ then
        there exists some $\eta > 0$ such that for any $t > 0$,
        \be
            \lim_{n \to \infty} \frac{1}{m_n^2} \ln \Pr\left(\lambda_1(A_n) > m_n t\right)
            > -(1-\eta)\frac{t^2}{4 \E \xi^2}.
        \ee
    \end{theorem}

    In particular, the assumptions of Theorem~\ref{thm:lower-bound-non-sharp} are satisfied for
    the (centered) Bernoulli random variable with $\pp > \frac 12$ mentioned above.

    For our applications to random graphs, we require a version of
    Theorem~\ref{thm:eigenvalue-ldp} for random bits chosen without
    replacement. Specifically, we consider the Erd\H{o}s-R\'enyi random graphs
    $\calG(n, m)$, where $m$ is an integer satisfying
    $|m - \pp \binom{n}{2}| = O(1)$ (and $\pp \in (0, 1)$ is fixed).
 
    \begin{theorem}\label{thm:eigenvalue-lpd-Gnm}
        Fix $\pp \in (0, 1)$
        and let $A_n$ be the centered adjacency matrix of a $\calG(n,m)$ random graph
        with $|m - \pp \binom{n}{2}| = O(1)$.
        For any integer $k \ge 1$ and any sequence $m_n$ satisfying $\sqrt n \ll m_n \ll n$,
        the sequence
        \be
            \frac{1}{m_n} (\sigma_1(A_n), \dots, \sigma_k(A_n))
        \ee
        satisfies an LDP with speed $m_n^2$ and good rate function $I: \R_+^k \to [0, \infty)$ given by
        $I(x) = \frac{|x|^2}{2} \cdot \frac{\ln \frac{1-\pp}{\pp}}{1-2\pp}$ (or $I(x) = |x|^2$ when $\pp = \frac 12$).

        If, in addition, $\pp \le \frac 12$ then the sequence
        \be
            \frac{1}{m_n} (\lambda_1(A_n), \dots, \lambda_k(A_n))
        \ee
        also satisfies an LDP with the same speed and rate function.
    \end{theorem}

 \section{Upper bound}

    The main observation is that in the regime we are interested in (namely,
    eigenvalues or singular values of order $\omega(\sqrt n)$), the probability
    of large eigenvalues can be controlled by a union bound over the potential
    eigenvectors; a similar observation was also used in~\cite{BhattacharyaGanguly20}.

    Let $\calM_k$ be the set of $n \times n$ matrices with rank at
    most $k$ and Frobenius norm at most 1. Let $\calM_k^+ \subset \calM_k$ consist
    of those matrices that are symmetric and positive semidefinite.

    \begin{lemma}\label{lem:eigenvalues-inner-product}
        For any symmetric matrix $A$,
        \be
            \left(\sum_{i=1}^k \max\{0, \lambda_i(A)\}^2\right)^{1/2} = \sup_{M \in \calM_k^+} \inr{A}{M}.
        \ee

        For any matrix $A$,
        \be
            \left(\sum_{i=1}^k \sigma_i(A)^2\right)^{1/2} = \sup_{M \in \calM_k} \inr{A}{M}.
        \ee
    \end{lemma}

    \begin{proof}
        To prove the first claim, assume without loss of generality that $\lambda_1(A) > 0$
        (if not, both sides are zero).
        Let $U D U^T = A$ be an eigen-decomposition of $A$ (where $D$ is
        diagonal and $U$ is orthogonal), and assume without loss of generality
        that the diagonal elements of $D$ are ordered as $\lambda_1(A) \ge \cdots \lambda_n(A)$.
        Let $\tilde D$ be the diagonal matrix with entries $\max\{0,
        \lambda_1(A)\}, \dots, \max\{0, \lambda_k(A)\},0, \dots, 0$, and define
        \be
            M = \frac{U \tilde D U^T}{\|\tilde D\|_F}
            = \frac{U \tilde D U^T}{\left(\sum_{i=1}^k \max\{0, \lambda_i(A)\}^2\right)^{1/2}}.
        \ee
        Then $M \in \calM_k^+$ and $\inr{A}{M} = \|\tilde D\|_F = \left(\sum_{i=1}^k \max\{0, \lambda_i(A)\}^2\right)^{1/2}$.
        This proves one direction of the first claim.

        For the other direction, take any $M \in \calM_k^+$, and decompose $A$ as
        $A_+ - A_-$, where $A_+$ and $A_-$ are positive semi-definite and the non-zero eigenvalues of $A_+$
        are the positive eigenvalues of $A$. Then
        \be
            \inr{A}{M} \le \inr{A_+}{M} \le \|A_+\|_F \|M\|_F
            \le \|A_+\|_F
            = \sqrt{\sum_{i=1}^k \lambda_i(A_+)^2}
            = \sqrt{\sum_{i=1}^k \max\{0, \lambda_i(A)\}^2}.
        \ee
        This proves the first claim. The proof of the second claim is identical, but uses
        a singular value decomposition instead of an eigen-decomposition.
    \end{proof}

    Hence, in order to prove the upper bounds in Theorem~\ref{thm:eigenvalue-ldp}, it suffices
    to control
    \be
        \Pr\left(\sup_{M \in \calM_k^+} \inr{A}{M} > t n^\alpha\right).
    \ee
    The first step is to replace the supremum with a finite maximum.

    \subsection{The net argument}

    \begin{definition}
        For a subset $\calN$ of a metric space $(X, d)$, we say that $\calN$ is an $\epsilon$-net of $X$
        if for every $x \in X$ there exists $y \in \calN$ with $d(x, y) \le \epsilon$.
    \end{definition}

    \begin{lemma}\label{lem:net-approximation}
        Let $\calN \subset \calM_k$ be an $\epsilon$-net (with respect to $\|\cdot\|_F$) for $\epsilon < \frac 12$.
        Then for any symmetric matrix $A$,
        \be
            \sup_{M \in \calM_k} \inr{A}{M} \le \frac{1}{1-2\epsilon} \sup_{N \in \calN} \inr{A}{N}.
        \ee
    \end{lemma}

    \begin{proof}
        Fix $M \in \calM_k$, and choose $N \in \calN$ with $\|N - M\|_F \le \epsilon$. Note that $N - M$
        has rank at most $2k$, and hence we can write $N - M = \epsilon M_0 + \epsilon M_1$ for some $M_0, M_1 \in \calM_k$.
        In other words, we can decompose
        \be
            M = N + \epsilon M_0 + \epsilon M_1
        \ee
        with $N \in \calN$ and $M_0, M_1 \in \calM_k$.
        It follows that
        \[
            \inr{A}{N} = \inr{A}{M} - \epsilon \inr{A}{M_0} - \epsilon \inr{A}{M_1} \ge \inr{A}{M}
            - 2\epsilon \sup_{M' \in \calM_k} \inr{A}{M'},
        \]
        and the claim follows.
    \end{proof}

    We have shown that to approximate the supremum it suffices to take a good enough net. In order to
    put this together with a union bound, we need a bound on the size of a good net.
    Such a bound can be found in~\cite[Lemma 3.1]{CandesPlan}.

    \begin{lemma}\label{lem:net-size}
        There is a constant $C$ such that for any $0 < \epsilon < 1$,
        there is an $\epsilon$-net (with respect to Frobenius norm) for $\calM_{k}$
        of size at most $(Ck/\epsilon)^{Cnk}$.
    \end{lemma}

    Applying a union bound over these nets gives the main result of this section: singular
    values and eigenvalues of $A$ can be controlled in terms of the deviations of linear functions
    of $A$. The main point here is that (as we will show in the next section) if $t \gg \sqrt n$
    then the $O(nk \ln \frac 1\epsilon)$ terms are negligible compared to the other terms.

    \begin{proposition}\label{prop:union-bound-conclusion}
        Let $A$ be a symmetric $n \times n$ random matrix with i.i.d.\ entries. For any integer $k \ge 1$,
        any $0 < \epsilon < \frac 12$, and any $t > 0$,
        \be
            \ln \Pr\left(
                \sum_{i=1}^k \sigma_i^2(A) > t
            \right)
            \le
            \sup_{M \in \calM_k} \ln \Pr\left(\inr{A}{M} \ge (1-2\epsilon) t\right)
            + O(nk \ln \frac 1\epsilon).
        \ee
    \end{proposition}

    \begin{proof}
        For the first inequality, let $\calN$ be an $\epsilon$-net for $\calM_k$ according to
        Lemma~\ref{lem:net-size}. By Lemma~\ref{lem:eigenvalues-inner-product}
        and Lemma~\ref{lem:net-approximation}
        \begin{align*}
            \Pr\left(\sum_{i=1}^k \sigma_i^2(A) > t\right)
            &= \Pr\left(\sup_{M \in \calM_k} \inr{A}{M} > t\right) \\
            &\le \Pr\left(\max_{N \in \calN} \inr{A}{N} > (1-2\epsilon) t\right).
        \end{align*}
        By a union bound,
        \begin{align*}
            \Pr\left(\max_{N \in \calN} \inr{A}{N} > (1-2\epsilon) t\right)
            &\le \sum_{N \in \calN} \Pr\left(\inr{A}{N} > (1-2\epsilon) t\right) \\
            &\le |\calN| \sup_{M \in \calM_k} \Pr\left(\inr{A}{M} > (1-2\epsilon) t\right),
        \end{align*}
        which, by our bound on $|\calN|$, completes the proof of the first claim.
    \end{proof}

    We remark that it is possible to prove a version of Proposition~\ref{prop:union-bound-conclusion}
    for eigenvalues also, giving an upper bound on $\Pr(\sum \lambda_i^2(A) > t)$ in terms of
    \be
            \sup_{M^+ \in \calM_k^+} \Pr\left(\inr{A}{M^+} \ge t\right).
    \ee
    This can in principle give a better bound on the eigenvalues than for the singular values.
    The issue is that
    we do not know how to exploit the additional information that we are testing $A$ against a positive
    semidefinite matrix.

    \subsection{Hoeffding-type argument}

    Using a Hoeffding-type argument, we can get a sharp upper bound on
    \be
        \sup_{M \in \calM_k} \ln \Pr\left(\inr{A}{M} \ge t\right)
    \ee
    for any $k$ and any $t$ (in fact, the sharp upper bound turns out not to depend on $k$).

    \begin{lemma}\label{lem:order-of-optimization}
        If $\xi$ is subgaussian then
        \be
            4 \sup_{s \in \R} \frac{\Lambda_\xi(s)}{s^2} = \left(\inf_{s \in \R} \frac{\Lambda_\xi^*(u)}{u^2}\right)^{-1} < \infty.
        \ee
    \end{lemma}

    \begin{proof}
        The fact that $\sup_{s \in \R} \frac{\Lambda_\xi(s)}{s^2} < \infty$ is the definition of subgaussianity.
        To show the claimed identity, let $L = \sup_{t \in \R}\frac{\Lambda_\xi(t)}{t^2}$
        and define $M_L(s) = L s^2$. Clearly, $\Lambda_\xi(s) \le M_L(s)$ for all $s \in \R$. It
        follows that $\Lambda_\xi^*(u) \ge M_L^*(u) = \frac{u^2}{4L}$; in other words,
        \be
            \frac{\Lambda_\xi^*(u)}{u^2} \ge \frac{M_L^*(u)}{u^2} = \frac{1}{4L}
        \ee
        for all $u$. This shows that
        \be
            4 \sup_{s \in \R} \frac{\Lambda_\xi(s)}{s^2} \ge
            \left(\inf_{u \in \R} \frac{\Lambda_\xi^*(u)}{u^2}\right)^{-1}.
        \ee
        For the other direction, suppose that for some $L'$ we have
        $\Lambda_\xi^*(u) \ge \frac{u^2}{4 L'} = M^{1/(4L)'}(u)$ for every $u$.
        Then (since $\Lambda_\xi$ is convex) $\Lambda_\xi(t) = \Lambda_\xi^{**}(t) \le M_{1/(4L')}^*(t) = L' t^2$
        for every $t$. The definition of $L$ ensures that $L' \ge L$, and this shows the other direction
        of the claim.
    \end{proof}

    \begin{proposition}\label{prop:hoeffding-conclusion}
        Let $\xi$ be a random variable with everywhere-finite moment-generating function, and define
        \be
            \Lambda_\xi(s) = \ln \E \exp(s \xi)
        \ee
        to be the cumulant-generating function of $\xi$. Let $A$ be a symmetric random matrix with
        zero diagonal, and with upper-diagonal elements distributed independently according to $\xi$.
        Define $\ell^* = \sup_{s > 0} \frac{\Lambda_\xi(s)}{s^2}$.
        Then
        \be
            \sup_{\|M\|_F \le 1} \Pr(\inr{A}{M} > t)
            \le \exp\left(-\frac{t^2}{8 \sup_{s > 0} \frac{\Lambda_\xi(s)}{s^2}}\right)
            = \exp\left(-\frac{t^2}{2} \inf_{s > 0} \frac{\Lambda_\xi^*(s)}{s^2}\right).
        \ee
    \end{proposition}

    \begin{proof}
        Since $\inr{A}{M} = \inr{A}{(M + M^T)/2}$ and since $\|(M + M^T)/2\|_F \le \|M\|_F$,
        it suffices to consider only symmetric matrices $M$.
        Let $m = \frac{n(n-1)}{2}$ and let $\xi_1, \dots, \xi_m$ be the upper-diagonal elements of $A$,
        in any order.
        Let $\|M\| \le 1$ be symmetric, with upper-diagonal entries $a_1, \dots, a_m$. Then
        $\inr{A}{M} = 2 \sum_{i=1}^m a_i \xi_i$, and so (for any $s > 0$)
        \begin{align*}
            \Pr(\inr{A}{M} > t)
            &= \Pr\left(\sum a_i \xi_i > t/2\right) \\
            &= \Pr\left(e^{s \sum a_i \xi_i} > e^{st/2}\right) \\
            &\le e^{-st/2} \E e^{s \sum a_i \xi_i} \\
            &= \exp\left(\sum_i \Lambda_\xi(s a_i) - st/2\right),
        \end{align*}
        where the inequality follows from Markov's inequality.
        Now, $\sum_{i=1}^m a_i^2 \le \frac{1}{2} \|M\|_F^2 \le \frac 12$, and so
        if we set $\ell^* = \sup_{r > 0} \frac{\Lambda_\xi(r)}{r^2}$ then
        \be
            \sum_i \Lambda_\xi(s a_i) = \sum_i \frac{\Lambda_\xi(s a_i)}{(s a_i)^2} (s a_i)^2
            \le s^2 \sum_i \ell^* a_i^2 \le \frac{s^2 \ell^*}{2}.
        \ee
        Hence,
        \be
            \Pr(\inr{A}{M} > t) \le \exp\left(\frac{s^2 \ell^*}{2} - \frac{st}{2}\right),
        \ee
        and the first claim follows by optimizing over $s$.

        The second claim follows immediately from Lemma~\ref{lem:order-of-optimization}.
    \end{proof}

    Putting Proposition~\ref{prop:hoeffding-conclusion} together with
    Proposition~\ref{prop:union-bound-conclusion}, we arrive at the following
    upper bound for singular values:

    \begin{corollary}\label{cor:upper-bound}
        Let $A$ be a symmetric $n \times n$ random matrix with i.i.d.\ upper diagonal entries.
        Assuming that the entries are subgaussian and have cumulant-generating function $\Lambda$,
        let $L = \inf_{s \in \R} \frac{\Lambda^*(s)}{s^2}$. Then
        for any integer $k$ and any $t > 0$, if $t^2 L > 2 nk$ then
        \be
            \ln \Pr\left(\sqrt{\sum_{i=1}^k \sigma_i^2(A)} > t\right) \le -\frac{t^2 L}{2} + O\left(n k \ln \frac{t^2 L}{nk}\right).
        \ee
    \end{corollary}

    \begin{proof}
        We combine 
        Proposition~\ref{prop:hoeffding-conclusion} and Proposition~\ref{prop:union-bound-conclusion},
        setting $\epsilon = \frac{nk}{t^2 L}$ (which is less than $\frac 12$ by assumption).
        This yields an upper bound of
        \be
            -\frac{t^2 L}{2} + O\left(nk + nk \ln \frac{t^2 L}{nk}\right),
        \ee
        and the $nk$ term can be absorbed in the final term.
    \end{proof}

    \begin{remark}\label{rem:inhomogeneous}
        Note that the argument leading to Corollary~\ref{cor:upper-bound} applies
        even when the entries $\xi_{ij}$ are not identically distributed as long as
        $L \le \inf_s \frac{\Lambda_{ij}^*(s)}{s^2}$ for every $i,j$, where $\Lambda_{ij}$
        is the cumulant-generating function of $\xi_{ij}$.
    \end{remark}

    \subsection{Lower bound}

    In this section, we give a lower bound that matches the upper bound of Corollary~\ref{cor:upper-bound}
    whenever $\sqrt n \ll t \ll n$.
    The starting point is the lower bound of Cram\'er's theorem~\cite[Theorem~27.3]{Kallenberg01}

    \begin{theorem}\label{thm:cramer}
        Let $\xi$ be a mean-zero random variable with everywhere-finite
        cumulant-generating function $\Lambda_\xi$.
        Let $\xi_1, \dots, \xi_m$ be independent copies of $\xi$. Then for any $t > 0$,
        \be
            \frac{1}{m} \ln \Pr\left(\sum_{i=1}^m \xi_i > mt\right) \to -\Lambda^*(t)
        \ee
        as $m \to \infty$.
    \end{theorem}

    \begin{proposition}\label{prop:lower-bound}
        In the setting of Corollary~\ref{cor:upper-bound}, suppose in addition that
        the function $s \mapsto s^{-2} \Lambda^*(s)$ achieves its minimum at some finite $s \in \R$.
        Then for any $1 \ll t \ll n^2$ and for any $w_1, \dots, w_k > 0$, we have
        \be
            \ln \Pr\left(\sum_{i=1}^k w_i \sigma_i(A_n) > |w| \sqrt t\right) \ge -\frac{t L}{2} - o(t).
        \ee
        (Here, $|w|$ denotes $\sqrt{\sum_i w_i^2}$.)
        If $s \mapsto s^{-2} \Lambda^*(s)$ achieves its minimum at some $s \ge 0$, then
        for any $1 \ll t \ll n^2$ and for any $w_1, \dots, w_k > 0$, we have
        \be
            \ln \Pr\left(\sum_{i=1}^k w_i \lambda_i (A_n) > |w| \sqrt t\right) \ge -\frac{t L}{2} - o(t).
        \ee
    \end{proposition}

    Choosing arbitrary $w_1, \dots, w_k$ and applying the Cauchy-Schwarz inequality,
    Proposition~\ref{prop:lower-bound} implies the same lower bounds on
    $\ln \Pr(\sum_i \sigma_i^2(A_n) > t)$
    and $\ln \Pr(\sum_i \lambda_i^2(A_n) > t)$. In particular, it really is a lower bound
    that matches the upper bound of Corollary~\ref{cor:upper-bound}.

    \begin{proof}
    Fix $t$ and assume that $\frac{\Lambda^*(s)}{s^2}$ achieves its minimum at $s_* \in \R$.
    Actually, we will assume $s_* \ne 0$; the case $s_* = 0$ is easily handled by replacing $s_*$ with
    $\epsilon > 0$ everywhere, and then sending $\epsilon \to 0$.
    Fix $w_1, \dots, w_k$ and assume $\sum_i w_i^2 = t$; because the statement of the proposition
    is homogeneous in $w$, this is without loss of generality. Now choose the
    smallest integers $\ell_1, \dots, \ell_k$
    so that $\ell_i - 1 \ge \frac{w_i}{|s_*|}$. We write $|\ell|^2$ for $\sum_i \ell_i^2$, and note
    that $|\ell|^2 \ge \frac{1}{s_*^2} \sum_i w_i^2 = \frac{t}{s_*^2}$, meaning that $1 \ll |\ell|^2 \ll n^2$.

    Let $M$ be a block-diagonal matrix, whose non-zero entries are all equal to $s_*$, appearing
    in blocks of size $\ell_i \times \ell_i$ for $i=1, \dots, k$. (The fact that $\sum_i \ell_i \le \sqrt k |\ell| \ll n$ implies
    that these blocks do indeed fit into an $n \times n$ matrix.)
    Then $M$ has rank $k$,
    and the singular values of $M$ are $|s_*| \ell_i$ for $i=1, \dots, k$;
    note that our choices of $\ell_i$ ensure that $w_i \le \sigma_i(M) \le w_i + 2 |s_*|$.
    Moreover, if we set $m = \sum_i \frac{\ell_i (\ell_i - 1)}{2}$ (which is also an integer, and counts
    the number of non-zero upper-diagonal elements of $M$) then
    $\inr{A}{M}$ is equal in distribution to $2 s_* \sum_{i=1}^m \xi_i$.
    Hence,
    \be
        \Pr\left(\inr{A}{M} > t\right)
        = \Pr\left(\sgn(s_*) \sum_{i=1}^m \xi_i > \frac{t}{2|s_*|}\right).
    \ee
    Now, $m = \frac 12 |\ell|^2 - \frac 12 \sum_i \ell_i$,
    while on the other hand
    \be
        \frac{t}{s_*^2} = \frac{\sum_i w_i^2}{s_*^2} \le \sum_i (\ell_i - 1)^2 = |\ell|^2 - 2 \sum_i \ell_i + 2k.
    \ee
    Since $\sum_i \ell_i \ge |\ell| \gg 1$, we have $\frac{t}{2s_*^2} \le m$ for sufficiently large $n$.
    Going back to our probability estimates, we have
    \begin{align*}
        \ln \Pr\left(\inr{A}{M} > t\right)
        &= \ln \Pr\left(\sgn(s_*) \sum_{i=1}^m \xi_i > \frac{t}{2|s_*|}\right) \\
        &\ge \ln \Pr\left(\sgn(s_*) \sum_{i=1}^m \xi_i > m |s_*|\right) \\
        &= -m \Lambda^*(s_*) + o(m) \\
        &= -\frac{t \Lambda^*(s_*)}{2 s_*^2} - o(t),
    \end{align*}
    where the second-last equality follows by Cram\'er's theorem (applied to the random variables $-\xi_i$
    in case $s_* < 0$).
    By von Neumann's trace inequality (see~\cite{Grigorieff})
    and the Cauchy-Schwarz inequality we have
    \be\label{eq:cs-for-singular-values}
        \inr{A}{M} \le \sum_{i=1}^k \sigma_i(A) \sigma_i(M) \le \sum_{i=1}^k \sigma_i(A) (w_i + 2 s_*) \\
        \le \sum_{i=1}^k \sigma_i(A) w_i + 2 s_* \sqrt k \sqrt{\sum_{i=1}^k \sigma_i^2(A)},
    \ee
    and hence
    \be
        \Pr\left(\inr{A}{M} > t\right) \le
        \Pr\left(\sum_{i=1}^k \sigma_i(A) w_i > t - t^{2/3}\right) + 
        \Pr\left(\sum_{i=1}^k \sigma_i^2(A) > \frac{t^{4/3}}{4 s_*^2 k}\right).
    \ee
    By Corollary~\ref{cor:upper-bound}, the second probability is of order $\exp(-\Omega(t^{4/3}))$,
    and hence
    \be
        \ln \Pr\left(\sum_{i=1}^k \sigma_i(A) w_i > t - t^{2/3}\right) \ge 
        (1-o(1)) \ln \Pr\left(\inr{A}{M} > t\right)
        \ge 
         -\frac{t \Lambda^*(s_*)}{2 s_*^2} - o(t).
    \ee
    Substituting in $t = |w| \sqrt t$ in place of $t - t^{2/3}$, the extra error term can be absorbed in the $o(t)$ term.

    For the second claim, simply note that if $s_* > 0$ then the matrix $M$ is positive semi-definite.
    Denoting $\lambda_i^+(A) = \max\{0, \lambda_i(A)\}$, we
    replace~\eqref{eq:cs-for-singular-values} by
    \be
        \inr{A}{M} \le \sum_{i=1}^k \lambda_i^+(A) \lambda_i(M) \le \sum_{i=1}^k \lambda_i^+(A) (w_i + 2 s_*)
        \le \sum_{i=1}^k \lambda_i^+(A) w_i + 2 s_* \sqrt k \sqrt{\sum_{i=1}^k \sigma_i^2(A)},
    \ee
    and the rest of the proof proceeds as before.
    \end{proof}

    There are a few extra useful facts that we can extract from the proof of
    Proposition~\ref{prop:lower-bound}, namely that we have explicit candidates for extremal
    eigenvectors and singular vectors. We will state these just for the largest eigenvector,
    but of course they also hold in other situations.

    \begin{corollary}\label{cor:smallest-eigenvector}
        Assume that $s \mapsto s^{-2} \Lambda^*(s)$ achieves its minimum at some $s_* \ge 0$.
        For $1 \ll t \ll n$, let $\ell = \lceil 1 +  t/s_* \rceil$ and define $v \in \R^n$
        by $v_1, \dots, v_\ell = s_*^{1/2} t^{-1/2}$ and $v_{\ell+1}, \cdots, v_n = 0$. Then
        $|v| \le 1 + o(1)$ and
        \be
            \ln \Pr(v^T A_n v \ge t) \ge -\frac{t^2 L}{2} - o(t^2).
        \ee
    \end{corollary}

    Corollary~\ref{cor:smallest-eigenvector} is immediate from the proof of Proposition~\ref{prop:lower-bound},
    because in the case $k = 1$ and $w_1 = \sqrt t$, the $M$ that we constructed in that proof is exactly $\sqrt t v v^T$.
    When we have extra quantitative control on the minimization of $\Lambda^*(s)/s$,
    it follows that the leading eigenvector must actually be close to the $v$ described above.
    We show this in Section~\ref{sec:conditional}, restricted for simplicity to the Bernoulli setting.

\subsection{The LDP}

Putting together Corollary~\ref{cor:upper-bound} and Proposition~\ref{prop:lower-bound},
we complete the proof of the LDP (Theorem~\ref{thm:eigenvalue-ldp}).
Take a sequence $m_n$ satisfying $\sqrt n \ll m_n \ll n$, and
set $X = \frac{1}{m_n} (\sigma_1(A_n), \dots \sigma_k(A_n))$.
Let $E \subset \R^k$
be any closed set, and let $t = \inf_{x \in E} |x|$. If $t > 0$ then $\frac{1}{m_n}
(\sigma_1(A_n), \dots, \sigma_k(A_n)) \in E$ implies that 
$\sum \sigma_i^2(A_n) > m_n^2 t^2$, and then Corollary~\ref{cor:upper-bound} implies that
\begin{multline*}
    \ln \Pr\left(X \in E\right)
\le \ln \Pr\left(\sum_{i=1}^k \sigma_i^2(A_n) > m_n^2 t^2\right) \\
        \le -\frac{m_n^2 t^2 L}{2} + O\left(n \ln \frac{m_n^2}{n}\right)
        = -\frac{m_n^2 t^2 L}{2} + o(m_n^2).
\end{multline*}
(And if $t = 0$ then the inequality above is trivially true.)

On the other hand, if $E \subset \R^k$ is open, then choose any $w \in E$. Since $E$
is open, there is some $\epsilon > 0$ so that if
$\inr{x}{w} \ge |w|^2$ and $|x|^2 \le |w|^2 + \epsilon$ then $x \in E$.
Now, Proposition~\ref{prop:lower-bound} implies that
\be
    \ln \Pr\left(\inr{X}{w} \ge |w|^2\right)
    = \ln \Pr\left(\sum_i \sigma_i(A_n) w_i \ge m_n |w|^2\right)
    \ge - \frac{m_n^2 |w|^2 L}{2} - o(m_n^2)
\ee
On the other hand, Corollary~\ref{cor:upper-bound} implies that
\begin{multline*}
    \ln \Pr\left(|X|^2 > |w|^2 + \epsilon\right)
    = \ln \Pr\left(\sum_i \sigma_i^2(A_n) \ge m_n^2 (|w|^2 + \epsilon)\right) \\
    \le - \frac{m_n^2 (|w|^2 + \epsilon) L}{2} - o(m_n^2).
\end{multline*}
In particular, $\Pr(|X|^2 > |w|^2 + \epsilon)$ is dominated by $\Pr(\inr{X}{w} \ge |w|^2)$,
implying that
\be
    \ln \Pr(X \in E) \ge 
    \ln \Pr\left(\inr{X}{w} \ge |w|^2 \text{ and } |X|^2 \le |w|^2 + \epsilon\right)
    \ge - \frac{m_n^2 |w|^2 L}{2} - o(m_n^2).
\ee
Since this holds for arbitrary $w \in E$, it implies the lower bound in the LDP.

The second part of Theorem~\ref{thm:eigenvalue-ldp} follows the exact same argument,
only it uses the second part of Proposition~\ref{prop:lower-bound}.

\subsection{The case of $\calG(n, m)$}
\label{sec:Gmn}

We next consider the case of Theorem \ref{thm:eigenvalue-lpd-Gnm}. The first
observation is that  $\pp \le \frac 12$ if and only if $\Lambda^*(s) / s^2$ achieves its
minimum at some non-negative $s$.

\begin{lemma}\label{lem:bernoulli-case}
    If $\xi = -\pp$ with probability $1-\pp$ and $\xi = 1-\pp$ with probability $\pp$ then
    $\Lambda^*$ (the convex conjugate of $\xi$'s cumulant generating function) satisfies
    \be
        \inf_{s \in \R} \frac{\Lambda^*(s)}{s^2} = \frac{\ln \frac{1-\pp}{\pp}}{1-2\pp},
    \ee
    and the minimum is uniquely attained at $s = 1 - 2\pp$.
\end{lemma}

\begin{proof}
    We recall that $\Lambda^*(s) = D(\pp + s, \pp)$ where
    \be
        D(r, \pp) = r \ln \frac{r}{\pp} + (1-r) \ln \frac{1-r}{1-\pp}
    \ee
    (with the convention that $D(r, \pp) = +\infty$ for $r \not \in (0, 1)$).
    Note that $D(r, \pp)$ is non-negative, convex, and has a double-root at $r=\pp$.
    Fix $\pp$ and define
    \be
        L(r) = \frac{D(r,\pp)}{(r-\pp)^2} = \frac{\Lambda^*(r-\pp)}{(r-\pp)^2}
    \ee
    (defined by continuity at $r=\pp$); our task is then to minimize $L$.
    We compute
    \be
        L'(r) = -\frac{(\pp + r) \ln \frac{r}{\pp} + (2 - \pp - r) \ln \frac{1-r}{1-\pp}}{(r-\pp)^3} =: -\frac{F(r)}{(r-\pp)^3}.
    \ee
    Then
    \begin{align*}
        F'(r) &= \ln \frac{r}{\pp} - \ln \frac{1-r}{1-\pp} + \frac{\pp}{r} - \frac{1-\pp}{1-r} \\
        F''(r) &= (r-\pp) \left(\frac{1}{r^2} - \frac{1}{(1-r)^2}\right).
    \end{align*}
    In particular, $F''$ has exactly two roots on $(0, 1)$: at $r=\frac 12$ and at $r=\pp$
    (counting with multiplicity in case $\pp=\frac 12$). It follows that $F$ has at most 4
    roots on $(0, 1)$. On the other hand, we can easily see that $F(\pp) = F'(\pp)
    = F''(\pp) = F(1-\pp) = 0$. Hence, $F(r)$ has a triple-root at $r=\pp$ and a
    single root at $r=1-\pp$, and no other roots. Since $r=\pp$ is only a
    triple-root, $L'(\pp) \ne 0$, and it follows that $r=1-\pp$ is the only root of $L'(r)$.
    It follows that $L(r)$ is minimized at either $r=0$, $r=1$, or $r=1-\pp$.
    The possible minimum values are therefore
    \be
        x := \pp^{-2} \ln \frac{1}{1-\pp},\qquad
        y :=(1-\pp)^{-2} \ln \frac{1}{\pp},\qquad \text{or }
        z := \frac{\ln \frac{1-\pp}{\pp}}{1-2\pp}.
    \ee

    We will show that $z$ is the smallest one. By symmetry in $\pp$ and $1-\pp$,
    it suffices to show that $z \le x$ for all $\pp$. Now,
    \be
        \pp^2(1-2\pp) (z - x) = \pp^2 \ln \frac{1-\pp}{\pp} + (1-2\pp) \ln (1-\pp) = (1-\pp)^2 \ln (1-\pp) - \pp^2 \ln \pp.
    \ee
    Let $f(\pp) = (1-\pp)^2 \ln(1-\pp) - \pp^2 \ln \pp$, and we need to show that $f(\pp) <
    0$ for $0 < \pp < \frac 12$ and $f(\pp) > 0$ for $\frac 12 < \pp < 1$. In fact,
    since $f(\pp) = - f(1-\pp)$, it suffices to show only one of these.
    Finally, note that $f(0) = f(\frac 12) = 0$, and $f''(\pp) > 0$ for $0 < \pp < \frac 12$,
    and it follows that $f(\pp) < 0$ for $0 < \pp < \frac 12$.
\end{proof}

Not only does Lemma~\ref{lem:bernoulli-case} establish the unique minimizer, it shows that
the behavior is locally quadratic around the minimizer. $D(q + s, q)$ is infinite outside the compact set $s \in [-q, 1-q]$, this also implies a quadratic lower bound on non-minimizers:
\begin{corollary}\label{cor:bernoulli-case}
    With the notation of Lemma~\ref{lem:bernoulli-case}, there is a constant $C = C(q)$ such that for every $s \in \R$,
    \[
        \frac{\Lambda^*(s)}{s^2} \ge \frac{\ln \frac{1-q}{q}}{1-2q} + C (s - (1 - 2q))^2.
    \]
\end{corollary}

To complete the proof of Theorem~\ref{thm:eigenvalue-lpd-Gnm},
it is enough to show that the upper bound of Corollary~\ref{cor:upper-bound}
and the lower bound of Proposition~\ref{prop:lower-bound} still hold in this setting;
then the proof of the LDP proceeds exactly as in the proof of Theorem~\ref{thm:eigenvalue-ldp}.
Checking Corollary~\ref{cor:upper-bound} is trivial: recalling that $A_n$ is the centered
adjacency matrix of $\calG(n, m)$ for $|m - \pp \binom{n}{2}| = O(1)$, we let
$\tilde A_n$ be the centered adjacency matrix of $\calG(n, \pp)$. Note that the distribution of
$A_n$ is equal to the distribution of $\tilde A_n$, conditioned on the event that $\tilde A_n$
has exactly $m$ positive entries on the upper diagonal; call this event $E$.
By Stirling's approximation, $\Pr(E) = \Omega(n^{-1})$, and it follows that for any event $F$,
\be
    \Pr(A_n \in F) = \Pr(\tilde A_n \in F \mid E) \le \frac{\Pr(A_n \in F)}{\Pr(E)} \le O(n \Pr(\tilde A_n \in F)).
\ee
In other words, $\ln \Pr(A_n \in F) \le \ln \Pr(\tilde A_n \in F) + O(\ln n)$,
and so Corollary~\ref{cor:upper-bound} immediately implies the same upper bound for $\calG(n, m)$.

For the lower bound, we need to look into the proof of Proposition~\ref{prop:lower-bound}.
Recall that in the proof of Proposition~\ref{prop:lower-bound}, we constructed a matrix $M$
with $O(t) = o(n^2)$ non-zero entries, all of which had the same value. For the $\calG(n, \pp)$
adjacency matrix $\tilde A_n$, $\inr{\tilde A_n}{M}$ has a (scaled and translated) binomial distribution;
for the $\calG(n, m)$ adjacency matrix $A_n$, $\inr{A_n}{M}$ has a (scaled and translated)
hypergeometric distribution.
Now, if $H_{k,n,r}$ denotes a hypergeometric random variable with population size $n$, $k$ successes,
and $r$ trials; and if $B_{\pp,r}$ denotes a binomial random variable with success probability $\pp$
and $r$ trials; then one easily shows using Stirling's approximation that
\be
    |\ln \Pr(H_{k,n,r} = s) - \ln \Pr(B_{k/n,r} = s)| = O(r^2/n).
\ee
In the setting of Proposition~\ref{prop:lower-bound}, the number of trials $r$ is the number of
non-zero elements in $M$, and since $r^2/n = O(t^2/n) = o(t)$, we have
\be
    \ln \Pr(\inr{A_n}{M} > t) \ge \ln \Pr(\inr{\tilde A_n}{M} > t) - o(t).
\ee
With this lower bound, we can follow the rest of the proof of Proposition~\ref{prop:lower-bound} 
to complete the proof of Theorem~\ref{thm:eigenvalue-lpd-Gnm}.

\section{Proof of Theorem~\ref{thm:lower-bound-non-sharp}}

Next, we consider the case that $\frac{\Lambda^*(s)}{s^2}$ does not achieve its
infimum at any $s > 0$, and we construct an example showing that taking $s \to 0$
does not yield the sharp bound. The basic idea is to use the first part
of Lemma~\ref{lem:eigenvalues-inner-product}, by producing a positive semi-definite
matrix $M$ and giving a lower bound on the tails of $\inr{A}{M}$. The main challenge
is to find a good matrix satisfying the positive definiteness constraint: in
Proposition~\ref{prop:lower-bound} we chose a matrix taking only one non-zero value,
specifically, $s_* \in \argmin \frac{\Lambda^*(s)}{s^2}$. The issue, of course, is that if $s_*$
is negative then such matrix cannot be positive semi-definite.
Instead, we will construct a rank-1 matrix taking four different non-zero values.

Consider a sequence $a_1, \dots, a_n$ whose non-zero elements take $m$ different values,
$\alpha b_1, \dots, \alpha b_m$, with $\alpha b_i$ repeated $\tilde m_i = \beta
m_i (1 + o(1))$ times respectively (the addition of the error term just allows us to deal
with the fact that matrices have integer numbers of rows and columns).
We will think of $m_i$ and $b_i$ as being fixed, while $\alpha$ and $\beta$ depend
on the tail bound that we want to show, with $\alpha$ being small and $\beta$ being large.
Then for any $t = \sum_{i=1}^m t_i$,
\be
    \Pr\left(\sum_i a_i \xi_i > t\right)
    \ge
    \prod_{i=1}^m
    \Pr\left(\sum_{j=1}^{\lceil \tilde m_i \rceil} \xi_j > t_i/(\alpha b_i)\right)
\ee
and so Theorem~\ref{thm:cramer} implies that if $\frac{t_i}{\alpha \beta m_i b_i} = \Theta(1)$
then
\begin{equation}\label{eq:non-tight-lower-bound-cramer}
    \ln \Pr\left(\sum_i a_i \xi_i > t\right)
    \ge - \beta \sum_i m_i \Lambda^*\left(\frac{t_i}{\alpha \beta m_i b_i}\right) - o\left(\beta \sum_i m_i\right).
\end{equation}
Our goal will be to choose the parameters $m_i, b_i, \alpha, \beta$, and $t_i$
to make the right hand side large. First, we will treat $m_i$ and $b_i$ as
given, and optimize over $t_i$, $\alpha$, and $\beta$.  We will enforce the
constraints $\sum_i t_i = t$ and $\sum_i a_i^2 = \alpha^2 \beta \sum_i m_i
b_i^2 = 2$.

Define
\begin{align*}
    \beta &= t^2 \frac{\sum_i m_i b_i^2}{2\left(\sum_i m_i b_i \Lambda'(b_i)\right)^2}, \\
    \alpha &= 2\left(\beta \sum_i m_i b_i^2\right)^{-1/2} = \frac{\sum_i m_i b_i \Lambda'(b_i)}{t \sum_i m_i b_i^2}, \text{ and }\\
    t_i &= \alpha \beta m_i b_i \Lambda'(b_i).
\end{align*}
With these choices, we have
\be
    \alpha^2 \beta = \frac{2}{\sum_i m_i b_i^2},
\ee
meaning that
\be
    \sum_i a_i^2 =\alpha^2 \beta \sum_i m_i b_i^2 = 2
\ee
and
\be
    \sum_i t_i = \alpha \beta \sum_i m_i b_i \Lambda'(b_i) = t.
\ee
(These turn out to be the optimal choices of $\alpha, \beta$, and $t$, although we do not need
to show this, since any choice will give us a bound.)
Plugging these parameters into~\eqref{eq:non-tight-lower-bound-cramer}, we obtain
\begin{equation}\label{eq:non-tight-lower-bound-substituted}
    \ln \Pr\left(\sum_i a_i \xi_i > t\right) \ge
    - \frac{t^2}{2} \cdot \frac{\sum_i m_i b_i^2 \cdot \sum_i m_i \Lambda^*(\Lambda'(b_i))}{\left(\sum_i m_i b_i \Lambda'(b_i)\right)^2}
    - o(t^2),
\end{equation}
where the $o(t^2)$ term depends on the parameters $m_i$ and $b_i$.

Next, we will define the parameters $m_i$ and $b_i$.
Take $\epsilon, \delta > 0$, and define
\begin{align*}
    m_1 &= \frac{1}{\epsilon^2},
        & b_1 &= \epsilon, \\
    m_2 &= 2\frac{\epsilon}{\delta^3},
        & b_2 &= -\delta, \\
    m_3 &= \frac{\epsilon^4}{\delta^6},
        & b_3 &= \frac{\delta^2}{\epsilon},
\end{align*}
and note that it is possible to define a positive semi-definite integral kernel taking the value
$b_i/2$ on a set of measure $2m_i$, simply by starting with a function taking the values $\sqrt{\epsilon}$
and $-\delta/\sqrt{\epsilon}$ on sets of size $1/\epsilon$ and $\epsilon/\delta^3$ respectively,
and then taking the outer product of that function with itself. It follows that if $\epsilon$
and $\delta$ are fixed and $\beta$ is large (and $\alpha$ is arbitrary),
then we can define a rank-1 p.s.d.\ matrix ($M$, say) with $(1 + o(1)) 2\beta m_i$ entries
taking the value $\alpha b_i/2$; note that $\|M\|_F^2 = \frac{1 + o(1)}{2}
\alpha \beta^2 \sum_i m_i = 1 + o(1)$. Since $A$ is a symmetric matrix with
$\xi$ on the upper diagonal,
this will yield
\be
    \inr{A}{M} = \sum_i a_i \xi_i,
\ee
where $(a_i)$ is a sequence containing $(1 + o(1)) \beta m_i$ copies of $\alpha b_i$.

We will first choose a small $\delta$
and then choose a smaller $\epsilon$. The error terms in the following analysis are taking this into
account, so for example we may write $\epsilon^2 \delta^{-k} = o(\epsilon)$ no matter how large $k$ is.
Our next task is to compute the various expressions in~\eqref{eq:non-tight-lower-bound-substituted},
in terms of $\epsilon$ and $\delta$. Before doing so, we observe some basic properties
of the Legendre transform.

\begin{lemma}\label{lem:large-term}
    Assume that $f$ is convex and differentiable
    and $\lim_{x \to \infty} \frac{f(x)}{x^2} = 0$.
    Then $\lim_{x \to \infty} \frac{f^*(f'(x))}{x^2} = 0$.
\end{lemma}

\begin{proof}
    Fix $x$ and let $y = f'(x)$. By the definition of $f^*$, we can write
    \be
        f^*(y) = \sup_z \{zy - f(z)\},
    \ee
    and note that the supremum is attained at $x = z$ (because the derivative is zero, and
    the expression being supremized is concave). Hence,
    \be
        f^*(f'(x)) = x f'(x) - f(x).
    \ee
    Convexity of $f$ implies that $f'$ is non-decreasing, and so
    $f(x) = o(x^2)$ implies that $f'(x) = o(x)$ as $x \to \infty$.
    Hence, $f^*(f'(x)) = x f'(x) - f(x) = o(x^2)$.
\end{proof}

\begin{lemma}\label{lem:small-terms}
    If $f$ is convex with $f(0) = f'(0) = 0$ and $f''(0) > 0$, and if both $f$ and $f^*$ are $\mathcal{C}^4$
    in a neighborhood of $0$, then
    \be
        f^*(f'(\epsilon))
        = f''(0) \frac{\epsilon^2}{2} + ((f^*)'''(0) (f'')^3(0) + 3 f'''(0)) \frac{\epsilon^3}{6}
        + O(\epsilon^4)
    \ee
    as $\epsilon \to 0$.
\end{lemma}

\begin{proof}
    This is nothing but Taylor's theorem and a computation. Setting $g = f^*$, we compute
    \be
        \frac{d}{d\epsilon} g(f'(\epsilon)) = g'(f'(\epsilon)) f''(\epsilon),
    \ee
    and then
    \be
        \frac{d^2}{d\epsilon^2} g(f'(\epsilon)) = g''(f'(\epsilon)) (f''(\epsilon))^2 + g'(f'(\epsilon)) f'''(\epsilon),
    \ee
    and finally
    \be
        \frac{d^3}{d\epsilon^3} g(f'(\epsilon)) = g'''(f'(\epsilon)) (f''(\epsilon))^3 + 3 g''(f'(\epsilon)) f''(\epsilon) f'''(\epsilon) + g'(f'(\epsilon)) f''''(\epsilon).
    \ee

    Our assumptions on $f$ ensure that $g'(0) = 0$, and hence the first-order term vanishes,
    the second-order term at $\epsilon=0$ becomes
    \be
        g''(0) (f''(0))^2,
    \ee
    and the third-order term at $\epsilon=0$ becomes
    \be
        g'''(0) (f''(0))^3 + 3g''(0) f''(0) f'''(0).
    \ee
    Finally, note that $g''(0) f''(0) = 1$.
\end{proof}

Note that $\Lambda$ satisfies the assumptions on $f$ in Lemmas~\ref{lem:large-term}
(because we assumed that $\Lambda(s) = o(s^2)$)
and~\ref{lem:small-terms} (because every cumulant-generating function defined on
a neighborhood of zero is $\mathcal{C}^\infty$ in a neighborhood of zero).
Note that $\Lambda$ and $\Lambda^*$ both have a second-order root at zero. Define
\be
    L = \Lambda''(0) > 0.
\ee
Expanding out the parameters
in~\eqref{eq:non-tight-lower-bound-substituted}, we have
\be
    \sum_i m_i b_i^2 = 1 + 2 \frac{\epsilon}{\delta} + \frac{\epsilon^2}{\delta^2}
\ee
for the first term in the numerator. The second term in the numerator is
\begin{align*}
    \sum_i m_i \Lambda^*(\Lambda'(b_i)) 
    &= \frac{1}{\epsilon^2} (\Lambda^* \circ \Lambda')(\epsilon)
    + 2 \frac{\epsilon}{\delta^3} (\Lambda^* \circ \Lambda')(-\delta)
    + \frac{\epsilon^4}{\delta^6} (\Lambda^* \circ \Lambda')(\delta^2/\epsilon).
\end{align*}
According to Lemma~\ref{lem:large-term} and our assumptions on $\Lambda$, the
last term is $o(\epsilon^2)$. Applying Lemma~\ref{lem:small-terms} to the other terms, we have
\begin{align*}
    \sum_i m_i \Lambda^*(\Lambda'(b_i))
    &= \frac{L}{2} + M \frac{\epsilon}{6} + L \frac{\epsilon}{\delta} - M \frac{\epsilon}{3} + O(\epsilon^2 + \epsilon \delta) \\
    &= \frac{L}{2}\left(1 + \frac{2\epsilon}{\delta}\right) - M \frac{\epsilon}{6} + O(\epsilon^2 + \epsilon \delta),
\end{align*}
where
\be
    M = (\Lambda^*)'''(0) L^3 + 3 \Lambda'''(0).
\ee

For the denominator in~\eqref{eq:non-tight-lower-bound-substituted}, we ignore the $i=3$ contribution,
giving a lower bound of
\begin{align*}
    \sum_i m_i \Lambda'(b_i)
    &\ge \frac{\Lambda'(\epsilon)}{\epsilon} - 2 \frac{\epsilon \Lambda'(-\delta)}{\delta^2} \\
    &= \Lambda''(0) + \frac{\epsilon}{2} \Lambda'''(0) + O(\epsilon^2)
        + 2 \frac{\epsilon}{\delta} \Lambda''(0) - \epsilon \Lambda'''(0) + O(\epsilon \delta) \\
    &= L \left(1 + 2\frac{\epsilon}{\delta}\right) - \frac{\epsilon}{2} \Lambda'''(0) + O(\epsilon^2 + \epsilon \delta).
\end{align*}
Putting everything together,
\begin{align*}
&\frac{\sum_i m_i \Lambda'(b_i) \cdot \sum_i m_i \Lambda^*(\Lambda'(b_i))}{\left(\sum_i m_i \Lambda'(m_i)\right)^2} \\
&= \frac{\left(1 + \frac{2\epsilon}{\delta} + O(\epsilon^2)\right) \left(\frac{L}{2} \left(1 + \frac{2\epsilon}{\delta}\right) - \frac{\epsilon M}{6} + O(\epsilon^2 + \epsilon \delta)\right)}
{\left(L\left(1 + \frac{2\epsilon}{\delta}\right) - \frac{\epsilon \Lambda'''(0)}{2} + O(\epsilon^2 + \epsilon \delta)\right)^2} \\
&= \frac{\frac{L}{2} - \frac{\epsilon M}{6} + O(\epsilon^2 + \epsilon \delta)}
{L^2 - \epsilon L \Lambda'''(0) + O(\epsilon^2 + \epsilon \delta)} \\
&= \frac{1}{2L} - \frac{\epsilon M}{6 L^2} + \frac{\epsilon \Lambda'''(0)}{2L^2} + O(\epsilon^2 + \epsilon \delta) \\
&= \frac{1}{2L} - \frac{\epsilon (\Lambda^*)'''(0) L}{6} + O(\epsilon^2 + \epsilon \delta),
\end{align*}
and in particular it is possible to choose $\delta$ and $\epsilon$ so that this quantity
is at most $(1-\eta) \frac{1}{2L}$ for some $\eta > 0$.

Going back to~\eqref{eq:non-tight-lower-bound-substituted} and recalling that the sequence $a_i$
can be realized as the elements of a rank-1 p.s.d. matrix, $M$ say, with $\|M\|_F = 1 + o(1)$,
we have shown that
\be
    \ln \Pr(\lambda_1(A_n) > t) \ge \ln \Pr(\inr{A}{M} > t \|M\|_F)
    \ge -(1-\eta) \frac{t^2}{4L} - o(t^2).
\ee
Replacing $t$ by $m_n t$ and recalling that $L = \Lambda''(0) = \E \xi^2$ completes the proof of Theorem~\ref{thm:lower-bound-non-sharp}.

\section{Back to cycle counts}

We now turn to the proofs of Theorems~\ref{thm:large-p-cycle-count} and \ref{thm:small-p-cycle-count}. The proofs are very similar, so we devote most of this section
to proving Theorem~\ref{thm:large-p-cycle-count} and then indicate what changes must be made 
to obtain Theorem~\ref{thm:small-p-cycle-count}.

Our eigenvalue LDP (Theorem~\ref{thm:eigenvalue-ldp}) allows us to
control the cycle-count contribution from a constant number of very extreme
eigenvalues, but in order to fully characterize the behavior of the cycle count,
For this, we will use
a deviation inequality from~\cite{GuionnetZeitouni00}:

\begin{theorem}\label{thm:bulk-deviation}
    Assume that $\|\xi\|_\infty < \infty$, and let $f: \R \to \R$ be a 1-Lipschitz,
    convex function.
    Define $X_n = \frac 1n \sum_{i=1}^n f(n^{-1/2} \lambda_i(A_n))$.
    Then there is a universal constant $C < \infty$ such that for any $\delta \gg n^{-1}$,
    \be
        \Pr(|X_n - \E X_n| \ge \delta) \le C \exp\left(-\frac{n^2 \delta^2}{C \|\xi\|_\infty^2}\right).
    \ee
\end{theorem}

Having controlled the bulk eigenvalues, we will
use Corollary~\ref{cor:upper-bound} to show
that the cycle count cannot be determined by $\omega(1)$ largish
eigenvalues.  Bear in mind that we will be applying our eigenvalue LDP to $\E A
- A$, where $A$ is the adjacency matrix, because
Theorem~\ref{thm:eigenvalue-lpd-Gnm} is for the positive eigenvalues of
centered matrices and we are interested in the negative eigenvalues here.

\subsection{The contribution of the bulk}

We consider two functions $f_1$ and $f_2$, where
\be
    f_1(x) = \begin{cases}
        0 & \text{if $x < 0$} \\
        x^k & \text{if $0 \le x < \sqrt{K}$} \\
        k K^{(k-1)/2} x - (k-1) K^{k/2} &\text{if $x \ge \sqrt{K}$}
    \end{cases}
\ee
and $f_2(x) = -f_1(-x)$. Then both $f_1$ and $f_2$ are $kK^{(k-1)/2}$-Lipschitz functions; also, $f_1$ is convex and $f_2$
is concave.

The following lemma is the main technical result of this section. Essentially,
it says that changing the cycle count using non-extreme eigenvalues carries a substantial entropy cost.
\begin{lemma}\label{lem:bulk-contribution}
    Let $A_n$ be the centered adjacency matrix of a $\calG(n,m)$ graph.
    There is a universal constant $C$ such that if $K \ge C$ then
    \be
        \Pr\left(\frac{1}{n^k}\sum_{i: \lambda_i(A_n) \ge -\sqrt {Kn}} \lambda_i^k(A_n) <  - t^k - C n^{-1}\right) \le \exp\left(-\Omega\left(\frac{n^k t^{2k}}{K^{k-1}}\right)\right).
    \ee
\end{lemma}

\begin{proof}
    We will prove the claim when $A_n$ is the centered adjacency matrix of a $\calG(n, p)$
    graph, with $p = m/\binom{n}{2}$. The result for $\calG(n, m)$ follows from the fact
    that a $\calG(n,m)$ graph can be obtained by starting from $\calG(n,p)$ and conditioning
    on the (probability $\Omega(1/n)$) event that there are exactly $m$ edges.

    Note that
    \be
        f_1(x) + f_2(x) \le \begin{cases}
            0 &\text{if $x < -\sqrt K$} \\
            x^k &\text{if $x \ge -\sqrt K$}.
        \end{cases}
    \ee
    Hence,
    \be
        \sum_i (f_1 + f_2)(n^{-1/2} \lambda_i(A_n)) \le n^{-k/2} \sum_{i: \lambda_i(A_n) \ge -\sqrt {Kn}} \lambda_i^k(A_n).
    \ee

    Since $-f_2$ is convex, Theorem~\ref{thm:bulk-deviation} applies to both $f_1$ and $f_2$,
    giving
    \be
    \label{eq:f_1-plus-f_2}
        \Pr\left(\frac 1n \tr[(f_1 + f_2)(n^{-1/2} A_n)] \le \frac 1n \E \tr [(f_1 + f_2)(n^{-1/2} A_n)] - s\right)
        \le 2 \exp(-\Omega(n^2 s^2/K^{k-1}))
    \ee
    whenever $s = \omega(K^{(k-1)/2}/n)$. Plugging in~\eqref{eq:f_1-plus-f_2} gives
    \be
        \Pr\left( \sum_{i: \lambda_i(A_n) \ge -\sqrt{Kn}} \lambda_i^k(A_n) \le n^{k/2} \E \tr (f_1 +  f_2)(n^{-1/2} A_n) - s \right)
        \le 2 \exp\left(-\Omega\left(\frac{s^2}{K^{k-1} n^k}\right)\right).
    \ee

    It remains to control $\E \tr[(f_1 + f_2)(n^{-1/2} A_n)]$; specifically, we want to
    show that $\E \tr (f_1 + f_2)(n^{-1/2} A_n)$
    is close to $n^{-k/2} \E \tr(A_n^k)$. But note that
    \begin{multline*}
        |\tr[(f_1 + f_2)(n^{-1/2} A_n) - n^{-k/2} A_n^k]| \\
        \le n^{-k/2} \sum_{i: |\lambda_i| > \sqrt{K n}} |\lambda_i(A_n)|^k \le n^{-k/2 + 1} |\sigma_1(A_n)|^k 1_{\{\sigma_1(A_n)| > \sqrt{K n}\}},
    \end{multline*}
    where $\sigma_1(A_n)$ is the largest singular value of $A_n$.
    Proposition~\ref{prop:hoeffding-conclusion} implies that if $K$ is sufficiently large
    then
    $\E [|\sigma_1(A_n)|^k 1_{\{|\sigma_1(A_n)| > \sqrt{K n}\}}] \le \exp(-\Omega(\sqrt n))$.
    Hence,
    \be
        \Pr\left( \sum_{i: \lambda_i(A_n) \ge -\sqrt{Kn}} \lambda_i^k(A_n) \le \E \tr (A_n^k) - s - \exp(-\Omega(\sqrt n))\right)
        \le 2 \exp\left(-\Omega\left(\frac{s^2}{K^{k-1} n^k}\right)\right).
    \ee
    Finally, note that $|\E \tr (A_n^k)| = O(n^{k-1})$ and set $s = t^k n^k$.
\end{proof}

We remark that the comparison between the exponents in Lemma~\ref{lem:bulk-contribution} and
the exponents in our eigenvalue LDP (Theorem~\ref{thm:eigenvalue-lpd-Gnm}) determines the range of deviations
to which our cycle-count deviation bounds apply: we get sharp bounds whenever Lemma~\ref{lem:bulk-contribution}
ensures that the bulk contribution is smaller than the contribution of the most extreme eigenvalues.
To that end, note that by Theorem~\ref{thm:eigenvalue-lpd-Gnm}, a single eigenvalue of order $-tn$ (which contributes $t^k$
to the $k$-cycle density) carries an entropy cost of order $t^2 n^2$; on the other hand, Lemma~\ref{lem:bulk-contribution}
shows that using the bulk eigenvalues to achieve the same $t^k$ change in the $k$-cycle density costs $t^{2k}n^k$ in entropy.
These costs cross over when $t^k \asymp n^{-\frac{k(2-k)}{2k-2}}$, but on the other hand applying Lemma~\ref{lem:bulk-contribution}
in this way also requires that $t^k \gg n^{-1}$. Therefore, we see that the bulk contribution is dominated by
the extreme eigenvalue contribution whenever
\[
    t^k \gg n^{-\min\{1, \frac{k(2-k)}{2k-2})\}},
\]
and the right hand side is $n^{-3/4}$ for $k = 3$ and $n^{-1}$ for $k \ge 5$. This computation determines
our critical exponent $c_*$ given in~\eqref{eq:c*-def}.

\subsection{Many large negative eigenvalues}

There is one situation that we still need to handle: the possibility that there
are $\omega(1)$ eigenvalues smaller than $-\Omega(\sqrt n)$, and $\omega(1)$
of these eigenvalues contribute to the triangle count.

The first observation is that although Corollary~\ref{cor:upper-bound} is written
for a fixed \emph{number} of singular values, it can be easily transferred to an
inequality for singular values above a certain threshold.

\begin{corollary}\label{cor:upper-bound-threshold-independent}
    With the notation of Corollary~\ref{cor:upper-bound}, if $\sigma_i = \sigma_i(A)$ are the
    singular values of $A$ then
    \be
        \ln \Pr\left(\sqrt{\sum_{\sigma_i > \sqrt{Kn}} \sigma_i^2} \ge t\right)
        \le -\frac{t^2 L}{2} + O\left(\frac{t^2}{K} \ln K\right)
    \ee
\end{corollary}

\begin{proof}
    Set $k = \lceil t^2/(Kn)\rceil$ and observe that
    if $\sigma_1, \dots, \sigma_k \ge \sqrt{Kn}$ then $\sum_{i=1}^k \sigma_i^2 \ge t^2$. Hence, we either have
    \be
        \sum_{\sigma_i > \sqrt{Kn}} \sigma_i^2 \le \sum_{i=1}^k \sigma_i^2,
    \ee
    or else $\sum_{i=1}^k \sigma_i^2 \ge t^2$. It follows that
    \be
        \ln \Pr\left(\sqrt{\sum_{\sigma_i > \sqrt{Kn}} \sigma_i^2} \ge t\right)
        \le
        \ln \Pr\left(\sqrt{\sum_{i=1}^k \sigma_i^2} \ge t\right),
    \ee
    and we conclude by applying Corollary~\ref{cor:upper-bound} with our choice of $k$.

    Finally, if $A$ is the centered adjacency matrix of a $\calG(n,m)$ graph then we use
    the same argument that was used to extend Corollary~\ref{cor:upper-bound} to the
    $\calG(n,m)$ case, namely that a $\calG(n,m)$ graph can be obtained by conditioning
    a $\calG(n,q)$ graph on an event of $\Omega(n^{-1})$ probability.
\end{proof}

Corollary~\ref{cor:upper-bound-threshold-independent} for extends to the case of
a $\calG(n,m)$ graph by the same argument that was used to extend Corollary~\ref{cor:upper-bound}
in the proof of Theorem~\ref{thm:eigenvalue-lpd-Gnm}.
Namely, a $\calG(n,m)$ graph can be obtained by conditioning
a $\calG(n,q)$ graph on an event of $\Omega(n^{-1})$ probability, and the extra factor $n$ introduced
by the conditioning
is of smaller order. Applying Corollary~\ref{cor:upper-bound-threshold-independent}
to a centered $\calG(n,q)$ adjacency matrix for $q = m/\binom n2$, and then applying Lemma~\ref{lem:bernoulli-case}
to get the explicit expression for $L$, we obtain the following bound:

\begin{corollary}\label{cor:upper-bound-threshold-conditioned}
    If $A$ is the centered adjacency matrix of a $\calG(n,m)$ random graph, let
    $p = m/\binom n2$ and let
    \[
        L = \frac{\ln \frac{1-p}{p}}{1 - 2p}.
    \]
    Let $\sigma_i = \sigma_i(A)$ be the singular values of $A$.
    Then for any fixed $K$, if $t \gg \sqrt n$
    \be
        \ln \Pr\left(\sqrt{\sum_{\sigma_i > \sqrt{Kn}} \sigma_i^2} \ge t\right)
        \le -\frac{t^2 L}{2} + O\left(\frac{t^2}{K} \ln K\right)
    \ee
\end{corollary}

\subsection{The upper bound in Theorem~\ref{thm:large-p-cycle-count}}

Let $A$ be the adjacency matrix of a $\calG(n, m)$ graph and
recall that $\tau_k(A) = \frac{\tr[A^k]}{n^k} + O(1/n)$.
Let $\tilde A = A - \E A$; by Corollary~\ref{cor:centered-cycle-count},
\begin{align}
    \Pr(\tau_k(A) \le p^k - t^k)
    &= \Pr(\tr[A^k] \le n^k p^k - n^k t^k + O(n^{k-1})) \notag \\
    &\le \Pr(\tr[\tilde A^k] \le - n^k t^k + O(n^{k-1})) + \Pr(\|\tilde A\|_{\text{op}} \ge \Omega(n)).
    \label{eq:centering-triangles}
\end{align}
Writing out $\tr[\tilde A^k] = \sum_i \lambda_i^k(\tilde A)$, choose $K = \omega(1)$ and $\epsilon = o(1)$ such that
$K^{k-1}/\epsilon^{2/k} = o(n^{k-2} t^{2k-2})$; this is possible because $t^k \gg n^{-c_*}$ implies that $n^{k-2} t^{2k-2} \gg 1$.
Applying Lemma~\ref{lem:bulk-contribution} to $\tilde A$ gives
\begin{equation}\label{eq:bulk-contribution}
\Pr\left(n^{-k} \sum_{i: \lambda_i \ge -\sqrt{Kn}} \lambda_i^k(\tilde A) < - \epsilon t^k\right)
\le \exp\left(- \Omega\left(\frac{\epsilon^{2/k} t^{2k} n^k}{K^{k-1}}\right)\right)
= \exp(-\omega(n^2 t^2)).
\end{equation}

On the other hand, Jensen's inequality implies that
\be
    \left| \sum_{i: \lambda_i < -\sqrt{Kn}} \lambda_i^k \right|
    \le \left( \sum_{i: \lambda_i < -\sqrt{Kn}} \lambda_i^2\right)^{k/2}
    \le \left( \sum_{i: \sigma_i > \sqrt{Kn}} \sigma_i^2\right)^{k/2},
\ee
where $\lambda_i = \lambda_i(\tilde A)$ and $\sigma_i = \sigma_i(\tilde A)$.
Recall here that $L = \inf_{s \in \R} \frac{\Lambda^*(s)}{s^2}$, where $\Lambda$ is the cumulant-generating
function of a centered Bernoulli random variable with success probability $p$.
Lemma~\ref{lem:bernoulli-case} (with $\pp = 1-p$) implies that
$L = \frac{\ln \frac{p}{1-p}}{2p - 1}$.
By Corollary~\ref{cor:upper-bound-threshold-conditioned} (and taking into account the fact that $\epsilon = o(1)$ and $K = \omega(1)$),
\begin{align*}
    \Pr\left(\sum_{i: \lambda_i < -\sqrt{Kn}} \lambda_i^k(\tilde A) < - (1-\epsilon) t^k n^k\right)
    &\le
    \Pr\left(\sqrt{\sum_{i: \sigma_i > \sqrt{Kn}} \sigma_i^2} > (1-\epsilon)^{1/k} t n\right) \\
    & \le \exp\left(
        -\frac{L}{2} t^{2} n^2 + o(t^{2} n^2)
    \right).
\end{align*}
Combined with~\eqref{eq:bulk-contribution}, this yields
\be
    \ln \Pr\left(\tr[\tilde A^k] \le -t^k n^k\right)
    \le - \frac{L t^{2} n^2}{2}(1 + o(1)).
\ee
Now we apply~\eqref{eq:centering-triangles}, noting that $n^k t^k = \omega(n^{k-1})$,
and so $n^k t^k + O(n^{k-1}) = n^k t^k(1 + o(1))$, to get
\be
    \ln \Pr(\tau_k(A) \le p^k - t^k)
    \le \max\left\{-\frac{L t^{2} n^2}{2}(1 + o(1)), \ln \Pr(\|\tilde A\|_{\text{op}} \ge \Omega(n))\right\}.
\ee
By Theorem~\ref{thm:eigenvalue-lpd-Gnm}, the second term in the maximum is of order $-\Theta(n^2)$ and so
the first term wins.

This completes the proof of the upper bound in Theorem~\ref{thm:large-p-cycle-count}
but let us also note two other facts that we can easily extract from the proof.
From~\eqref{eq:bulk-contribution} we see that only the extremely negative eigenvalues contribute
to the cycle deviation:

\begin{corollary}\label{cor:small-eigenvalues-dominate}
    Conditioned on $\tau_k(A) \le p^k - t^k$,
    $\sum_{i: \lambda_i \le -\Omega(\sqrt n)} \lambda_i^k(\tilde A) \le -t^k n^k (1 - o(1))$
    with high probability.
\end{corollary}

The other piece of information we can extract from our proof is that the vertex degrees
of a cycle-deficient graph are close to constant.

\begin{corollary}\label{cor:degree-profile}
    Conditioned on $\tau_k(A) \le p^k - t^k$, if $d_1, \dots, d_n$ are the vertex degrees of the graph
    then with high probability
    \be
        \sum_i (d_i - pn)^2 = o(t^k n^3).
    \ee
\end{corollary}

\begin{proof}
    In the proof of the upper bound of Theorem~\ref{thm:large-p-cycle-count}, recall
    that $\ln \Pr(\|\tilde A\|_{\text{op}} \ge \Omega(n)) \ll \ln \Pr(\tau_k(A) \le p^k - t^k)$,
    and it follows that conditioned on $\tau_k(A) \le p^k - t^k$ we have $\|\tilde A\|_{\text{op}} = o(n)$
    with high probability.
    Since $\Pr(\tau_k(A) \le p^k - t^k(1 + \epsilon)) \ll \Pr(\tau_k(A) \le p^k - t^k)$, we also
    have $\tau_k(A) \ge p^k - t^k(1 + \epsilon)$ with high conditional probability.
    But on the event that $\|\tilde A\|_{\text{op}} \le \epsilon n$ and $\tau_k(A) \ge p^k - t^k(1 + \epsilon)$,
    Lemma~\ref{lem:centered-cycle-count} implies that
    \[
        n^{k-3} \sum_i (d_i - pn)^2 \le \epsilon t^k n^k,
    \]
    and the claim follows.
\end{proof}

\subsection{The lower bound in Theorem~\ref{thm:large-p-cycle-count}}

The idea here is to partition the adjacency matrix into blocks, and then consider
the event that certain prescribed numbers of edges are present in each block. By choosing
all parameters correctly, we can ensure that this event has the correct probability, and
also that on this event the cycle density will behave as desired.

Recall that $L = \sup_s \frac{\Lambda^*(s)}{s^2}$ and that $s_* = 2p-1$ is the maximizing value of $s$.
Let $\ell$ be the closest integer to $t n / s_*$ and
let $\xi_1, \dots, \xi_{\binom n2}$ be some ordering of the upper diagonal of $\tilde A$.
Let $U_{11}$ be the collection of $i$ for which $\xi_i$ is in the upper-left $\ell \times \ell$ 
submatrix; let $U_{12}$ be the collection of $i$ for which $\xi_i$ is in the upper-right $\ell \times (n - \ell)$
submatrix; and let $U_{22}$ be the remaining indices.
Define $z$ by $\ell = z n$, and note that $z = (1 + O(1/n)) t / s_*$.
Now let $S_* = \lfloor s_* |U_{11}| \rfloor$
and $T_* = \lceil z |U_{12}| s_* \rceil$, and
let $\Omega$ be the event that
\begin{align}
    \frac{1}{|U_{11}|} \sum_{i \in U_{11}} \xi_i &= S_* \label{eq:omega-1} \\
    \frac{1}{|U_{12}|} \sum_{i \in U_{12}} \xi_i &= T_* \label{eq:omega-2}
\end{align}

We claim that
$\ln \Pr(\Omega) \ge -\frac{t^{2} n^2 L}{2}(1 + o(1))$,
and that conditioned on $\Omega$, $\tau_k(G) \le p^k - t^k$ with non-negligible probability. Together, these
imply the lower bound of Theorem~\ref{thm:large-p-cycle-count}.

\begin{lemma}
\[
\ln \Pr(\Omega) \ge -\frac{t^{2} n^2 L}{2}(1 + o(1)).
\]
\end{lemma}

\begin{proof}
Let $\Omega_1$ be the event of~\eqref{eq:omega-1} and let $\Omega_2$ be the event of $\eqref{eq:omega-2}$.
These events can be described simply in terms of hypergeometric random variables. Indeed,
$\sum_{i \in U_{11}} (\xi_i + p)$ is a hypergeometric random variable with $\binom{\ell}{2}$ trials,
and a population of size $\binom{n}{2}$ containing $m$ successes; therefore $\Omega_1$ is just the
event that this hypergeometric variable takes a particular value. Conditioned on $\Omega_1$,
$\sum_{i \in U_{12}} (\xi_i + p)$ is a hypergeometric random variable with $\ell (n-\ell)$ trials,
and a population of size $\binom{n}{2} - \binom{\ell}{2}$ containing $m - S_*$ successes; the event $\Omega_2$
is just the event that this hypergeometric variable takes a particular value.

These hypergeometric probabilities can be computed explicitly; we will make use of the approximation
that comes simply from applying Stirling's approximation to the explicit computation
(see, e.g.,~\cite[Lemma~2.1.33]{DemboZeitouni}): if $Z$ is a hypergeometric random variable with
$r$ trials from a population of size $R$ with $\alpha R$ successes, then for any integer $b$ in the range of $Z$,
\begin{equation}\label{eq:ldp-hypergeometric}
    \frac 1r \ln \Pr(Z = b) = - D(b/r, \alpha) - \frac{1 - r/R}{r/R} D\left(\frac{\alpha - b/R}{1 - r/R}, \alpha\right) + O\left(\frac{\ln R}{r}\right)
\end{equation}
where $D(q + s, q) = (q + s) \ln \frac {q + s}{q} + (1-q - s) \ln \frac{1-q - s}{1-q}$ is, as before,
the Legendre transform of a centered Bernoulli variable's cumulant generating function.

Applying~\eqref{eq:ldp-hypergeometric} to $\Omega_1$, since $\ln n \ll \ell \ll n$ and since $D(\alpha + \epsilon, \alpha) = \Theta(\epsilon^2)$, we have
\[
    \frac{1}{|U_{11}|} \Pr(\Omega_1) \to -D(p + s^*, p) = -L,
\]
and hence
\[
    \ln \Pr(\Omega_1)
= -(1 + o(1)) \frac{t^{2} n^2 L}{2}.
\]

Since $\Omega = \Omega_1 \cap \Omega_2$, it suffices to show that
\[
\Pr(\Omega_2 \mid \Omega_1) = \exp(-o(t^{2} n^2)).
\]
Recall that conditioned on $\Omega_1$, $\sum_{i \in U_{12}} (\xi_i + p)$ is hypergeometric with $\Theta(z n^2)$
trials, a population size of $\Theta(n^2)$, and $m - O(z^2)$ successes.
The event $\Omega_2$ is asking for this hypergeometric variable to deviate from its mean (which is
of order $\Theta(z n^2)$) by a fixed quantity of smaller order, namely $\Theta(z^2 n^2)$. By~\eqref{eq:ldp-hypergeometric},
\[
    \frac{1}{|U_{12}|} \ln \Pr(\Omega_2 \mid \Omega_1) = -D(p + \Theta(z), p) + o(z) = -o(z).
\]
Therefore,
$ \Pr(\Omega_2 \mid \Omega_1) = \exp(-o(z^2 n^2)) = \exp(-o(t^{2} n^2))$.
\end{proof}

Next, we show that conditioned on $\Omega$, $G$ has fewer cycles.
For ease of notation, let us first describe the conditional distribution of $G$ given $\Omega$
in terms of different parameters.
Let $\ell = zn$ for $n^{-{2/3}} \ll z \ll 1$, and fix $0 < q < p$. Consider a random graph $G$
with $m = p \binom n2$ edges, and let $A$ be its adjacency matrix.
Suppose that $(p - q) \binom{\ell}{2}$ of these edges
are uniformly distributed on the upper diagonal of the top-left $\ell \times \ell$
block of $A$, $(p + \frac{z}{1-z} q) \ell (n-\ell) + O(1)$ are uniformly distributed on the
top-right $\ell \times (n - \ell)$ block, and $(p - \frac{z^2}{(1-z)^2} q) \binom{n - \ell}{2} + O(1)$
are uniformly distributed on the remaining part of the upper-diagonal.
The $O(1)$ error terms ensure that it is possible to satisfy the constraints with
integer numbers of edges, and these error terms are also compatible
with the requirement that there are $p \binom{n}{2}$ edges in total.
Finally, note that the distribution of $G$ conditioned on $\Omega$ is the
same as the distribution we have described above, for some $q$ within $\Theta(1/\ell^2)$ of $s_*$.

\begin{lemma}\label{lem:conditioned-cycle-density}
    For the random graph $G$ described above,
    \[
        \E \tau_k(G) = p^k - z^k q^k + o(z^k)
    \]
    and
    \[
        \Var(\tau_k(G)) = O(n^{-2}).
    \]
    In particular, if $z^k q^k = \omega(n^{-1})$ then $\tau_k(G) \le p^k - z^k q^k + o(z^k)$ w.h.p.
\end{lemma}

Recalling that $z^k q^k \ge (1 + o(1)) t^k n^k$, Lemma~\ref{lem:conditioned-cycle-density}
completes the proof of the lower bound of Theorem~\ref{thm:large-p-cycle-count}, after
replacing $t$ by $(1 - o(1)) t$.

\begin{proof}
    To compute the expected number of cycles, let $B$ the the $n \times n$ block matrix that agrees with $\E A$
    except on the diagonal. That is, $B$ takes the value $p - q$
    on the top-left $\ell \times \ell$ block, the value $(p + \frac{z}{1-z} q) + O(n^{-2} z^{-1})$ on the top-right $\ell \times (n-\ell)$ block,
    and the value $(p - \frac{z^2}{(1-z)^2} q) + O(n^{-2})$ on the bottom $(n - \ell) \times (n-\ell)$ block.
    Then $B$ has rank-2, and it is well-approximated by the rank-2 matrix $p \1 - q w w^T$, where $w$ has ones in the
    first $zn$ entries, and takes the value $-z/(1-z)$ in the other entries. More precisely,
    \[
        \|B - (p \1 - q w w^T)\|_F = O(n^{-1} z^{-1/2}) = o(1),
    \]
    with the main contribution coming from the $\Theta(z n^2)$ entries of size $O(z^{-1} n^{-2})$.
    Since $\1$ and $w$ are orthogonal, $p \1 - q w w^T$ has eigenvalues $p n$ and $-q |w|^2 = -q z n + O(z^2 n)$.
    By Weyl's eigenvalue inequalities, $B$ has eigenvalues $p n + o(1)$ and $-q z n + o(1)$.
    Therefore $\tr B^k = p^k n^k - q^k z^k n^k + O(n^{k-1})$.

    Next, consider $\tr [(\E A)^k]$. Recalling that $\E A$ agrees with $B$ except on the diagonal (because $\E A$
    is zero on the diagonal and $B$ is not), we have $\|\E A - B\|_{\text{op}} = O(1)$, and so Weyl's eigenvalue
    inequalities imply that $\E A$ has an eigenvalue of $p n + O(1)$, an eigenvalue of $-q z n + O(1)$,
    and its remaining eigenvalues are bounded. Therefore, $\tr[(\E A)^k] = p^k n^k - q^k z^k n^k + O(n^{k-1})$.

    To compare $\tr[(\E A)^k]$ to $\E \tau_k(G)$, expand $\tr[(\E A)^k]$ in terms of closed walks of length $k$:
    let $\Gamma_k$ be the set of $(k+1)$-tuples $v_1, \dots, v_{k+1}$ with $v_{k+1} = v_1$ and $v_i \ne v_{i+1}$ for all $i$.
    Then
    \begin{equation}\label{eq:trace-expansion}
        \tr[(\E A)^k] = \sum_{(v_1, \dots, v_{k+1}) \in \Gamma_k} \prod_{i=1}^k (\E A)_{v_i,v_i+1}.
    \end{equation}
    Let $\tilde \Gamma_k$ be the set of $(k+1)$-tuples $v_1, \dots v_{k+1}$ in $\Gamma_k$ such that $v_1, \dots, v_k$
    are distinct. Then $|\Gamma_k - \tilde \Gamma_k| = O(n^{k-1})$. Since each summand in~\eqref{eq:trace-expansion}
    is bounded,
    \[
        \tr[(\E A)^k] = \sum_{(v_1, \dots, v_{k+1}) \in \tilde \Gamma_k} \prod_{i=1}^k (\E A)_{v_i,v_i+1} + O(n^{k-1}).
    \]
    On the other hand,
    \[
        \binom{n}{k} \E \tau_k(G) = \sum_{(v_1, \dots, v_{k+1}) \in \tilde \Gamma_k} \Pr(\{v_i, v_{i+1}\} \in E(G) \text{ for all $i$}).
    \]
    For each $i$, $\Pr(\{v_i, v_{i+1}\} \in E(G)) = (\E A)_{v_i,v_{i+1}}$. Because the edges are chosen without replacement
    these terms are not independent. However, we always have the inequality
    \[
        \Pr(\{v_i, v_i+1\} \in E(G) \mid \{v_1, v_2\} \in E(G), \dots, \{v_{i-1}, v_i\} \in E(G)
        \le (\E A)_{v_i, v_{i+1}}.
    \]
    Therefore,
    \[
        \binom{n}{k} \E \tau_k(G) \le \sum_{(v_1, \dots, v_{k+1}) \in \tilde \Gamma_k} \prod_{i=1}^k (\E A)_{v_i,v_i+1} = \tr[(\E A)^k] + O(n^{k-1}) = p^k n^k - q^k z^k n^k + O(n^{k-1}).
    \]
    This proves the claim about the expectation.

    Next, we consider the variance of the cycle density.
    For an ordered $k$-tuple $S \subset V(G)$, let $T_S$ be the event that the vertices in $S$ form a $k$-cycle.
    Note that because the edges are drawn
    without replacement, if $S_1$ and $S_2$ do not share an edge then $T_{S_1}$ and $T_{S_2}$ are non-positively
    correlated.
    Therefore,
    \[
        \Var(T(G)) = \sum_{S_1,S_2} \Cov(T_{S_1}, T_{S_2})
        = \sum_{S_1, S_2: |S_1 \cap S_2| \ge 2} \Cov(T_{S_1}, T_{S_2}).
    \]
    There are at most $n^{2k-2}$ elements in the sum, and each is bounded by 1. Therefore,
    $\Var(T(G)) \le n^{2k-2}$ and so $\Var(\tau(G)) = O(n^{-2})$.
\end{proof}

\subsection{The two extreme eigenvalues}

In proving the upper bound on $\Pr(\tau_k(A) \le p^k - t^k)$, we applied the inequality
$\sum_i |a_i|^k \le (\sum_i a_i^2)^{k/2}$ to the collection of most-negative eigenvalues.
In order to understand how these most negative eigenvalues are actually distributed,
observe that in order for the inequality above to be an equality, all but one of the terms
in the sum must be zero. Made quantitative, this observation implies that in order for our
probability upper bound to be tight, the smallest eigenvalue must dominate the others.
In what follows, we write $\|a\|_p^p$ for $\sum_i |a_i|^p$.

\begin{lemma}\label{lem:l3-l2}
    Let $a_1, \dots$ be a sequence of non-negative numbers, in non-increasing order. For
    $\epsilon > 0$ and $k \ge 3$, if
    \be
        \sum_{i \ge 2} a_i^k \ge \epsilon a_1^k
    \ee
    then
    \be
        \|a\|_2^2 \ge (1 + \epsilon)^{1/k} \|a\|_k^2.
    \ee
\end{lemma}

\begin{proof}
    If $\sum_{i \ge 2} a_i^k \ge \epsilon a_1^k$ then $\|a\|_\infty^k = a_1^k \le \frac{\|a\|_k^k}{1 + \epsilon}$.
    Then $\|a\|_k^k \le \|a\|_\infty^{k-2} \|a\|_2^2 \le (1 + \epsilon)^{-(k-2)/k} \|a\|_k^{k-2} \|a\|_2^2$,
    and the claim follows.
\end{proof}

Applying Lemma~\ref{lem:l3-l2} to the most negative eigenvalues of $\tilde A$ allows
us to show that the eigenvalues of $\tilde A$ satisfy the claims that
Theorem~\ref{thm:large-p-cycle-count} makes for the eigenvalues of $A$.

\begin{corollary}\label{cor:extreme-eigenvalues-of-tilde-A}
    In the setting of Theorem~\ref{thm:large-p-cycle-count},
    for any $\epsilon > 0$, conditioned on $\tau_k(A) \le p^k - t^k$ we have
    \be
        \lambda_n^k(\tilde A) \le -(1-\epsilon) t^k n^k \text{ and }
        \lambda_{n-1}^k(\tilde A) \ge -\epsilon t^k n^k
    \ee
    with high probability.
\end{corollary}

\begin{proof}
    Let $S = \{i: \lambda_i(\tilde A) \le -\Omega(\sqrt n)\}$.
    By Corollary~\ref{cor:small-eigenvalues-dominate}, for any $\delta > 0$, conditioned on $\tau_k(A) \le p^k - t^k$ we have
    \be
        \sum_{i \in S} \lambda_i^k(\tilde A) \le -(1 - \delta) t^k n^k
    \ee
    with high probability. On this event, we either have $\lambda_n^k(\tilde A) \le -(1-\delta - \epsilon) t^k n^k$
    or $\sum_{i \in S \setminus \{n\}} \lambda_i^k(\tilde A) \le -\epsilon t^k n^k$.
    We will show that for some $\delta = \Omega(\epsilon)$,
    \be
    \Pr \left(\sum_{i \in S} \lambda_i^k(\tilde A) \le -(1 - \delta) t^k n^k
    \text{ and } \lambda_n^k(\tilde A) > -(1 - \delta - \epsilon) t^k n^k
    \text{ and }\sum_{i \in S \setminus \{n\}} \lambda_i^k(\tilde A) \le -\epsilon t^k n^k\right)
    \ee
    is much smaller than $\Pr(\tau_k(A) \le p^k - t^k)$; this will imply the claim.

    Indeed, applying Lemma~\ref{lem:l3-l2} to the sequence of $|\lambda_i|$ for $i \in S$, we see
    that if
    \be
        \sum_{i \in S} \lambda_i^k(\tilde A) \le -(1 - \delta) t^k n^k
        \text{ and } \lambda_n^k(\tilde A) > -(1 - \delta - \epsilon) t^k n^k
        \text{ and }\sum_{i \in S \setminus \{n\}} \lambda_i^k(\tilde A) \le -\epsilon t^k n^k
    \ee
    then
    \be
        \sum_{i \in S} \lambda_i^2(\tilde A) \ge (1 + \epsilon)^{1/k} (1-\delta) t^{2} n^2
        \ge (1 + \Omega(\epsilon)) t^{2} n^2,
    \ee
    where the last inequality follows by choosing a sufficiently small $\delta = \Omega(\epsilon)$.
    But Corollary~\ref{cor:upper-bound-threshold-conditioned} implies that
    \begin{align*}
        \Pr\left(\sum_{i \in S} \lambda_i^2(\tilde A) \ge (1 + \Omega(\epsilon)) t^{2} n^2\right)
        &\le \exp\left(- (1 + \Omega(\epsilon)) (1 - o(1)) \frac{t^{2} n^2 L}{2}\right) \\
        &= o(\Pr(\tau_k(A) \le p^k - t^k)),
    \end{align*}
    where the final bound follows from the lower bound of Theorem~\ref{thm:large-p-cycle-count}.
\end{proof}

Note that although we have been focussing on the smallest (i.e.\ negative, with large magnitude) eigenvalues,
this same argument tells us about the largest eigenvalues also: if $\lambda_1(\tilde A) \ge \epsilon^{1/k} t n$,
then in order to have $\sum_i \lambda_i^k(\tilde A) \le -(1 - o(1)) t^k n^k$ we would need
$\sum_{i = 2}^n \lambda_i^k (\tilde A) \le -(1 - \epsilon - o(1)) t^k n^k$, which by the argument above has
probability $\exp(- (1 + \Omega(\epsilon)) L t^{2} n^2/2) = o(\Pr(\tau_k(A) \le p^k - t^k))$. Therefore,
we obtain the following bound on the largest eigenvalue:

\begin{corollary}\label{cor:conditional-largest-eigenvalue}
    In the setting of Theorem~\ref{thm:large-p-cycle-count},
    conditioned on $\tau_k(A) \le p^k - t^k$, with high probability $\lambda_1(\tilde A) = o(t n)$.
\end{corollary}

To complete the proof of Theorem~\ref{thm:large-p-cycle-count}, we need to
pass from the eigenvalues of $\tilde A$ to the eigenvalues of $A$; recall that
$A = \tilde A + p \1 - p I$. Since $p \1 \ge 0$, we have
\be
    \lambda_{n-1}(A) \ge \lambda_{n-1}(\tilde A) - p,
\ee
and so $\lambda_{n-1}(\tilde A) \ge -o(t n)$ implies the same for $\lambda_{n-1}(A)$.
For $\lambda_n$, let $v$ be a unit eigenvector of $\tilde A$ with eigenvalue $\lambda_n(\tilde A)$.
By Corollary~\ref{cor:degree-profile}, with high (conditional on $\tau_k(A) \le p^k - t^k$) probability,
$|\tilde A 1|^2 = o(t^k n^3)$, where $1$ denotes the all-ones vector. On this event, expand $1$ in the basis of eigenvectors of $\tilde A$
to see that $|\tilde A 1|^2 \ge \inr{1}{v}^2 \lambda_n(\tilde A)^2$. Therefore
$\inr{1}{v}^2 \le o(t^k n^3 \lambda_n(\tilde A)^{-2}) = o(t^{k-2} n) = o(tn)$.
Now, $\inr{A}{v v^T} \le \inr{\tilde A}{v v^T} + p \inr{1}{v}^2 = \lambda_n(\tilde A) + o(t n)$
and by considering $v$ as a potential eigenvector of $A$, it follows that $\lambda_n(A) \le \lambda_n(\tilde A) + o(t n)$.
This completes the proof of Theorem~\ref{thm:large-p-cycle-count}.

\subsection{Theorem \ref{thm:small-p-cycle-count}}

Like Theorem \ref{thm:large-p-cycle-count}, Theorem \ref{thm:small-p-cycle-count} has 
three elements: 
an upper bound on the probability of a moderate deviation, a lower bound, and a bound
on the most negative eigenvalue of the adjacency matrix. 

The upper bound is proved exactly as in
the proof of Theorem \ref{thm:large-p-cycle-count}. The singular values of the eigenvalues are
controlled by the rate function involving $\inf_{s \in \R}\frac{\Lambda^*(s)}{s^2}$, which we have
already established to be $\frac{\ln \frac{1-p}{p}}{2(1-2p)}$. Upper bounds on singular values 
then give upper bounds on eigenvalues. The entire argument is independent of whether $p \ge \frac12$
or $p < \frac12$. 

The proof of the lower bound in Theorem~\ref{thm:small-p-cycle-count} is similar to that
of the lower bound in Theorem~\ref{thm:large-p-cycle-count}, except that we use the vector
$v = (\frac 1{\sqrt n}, \dots, \frac 1{\sqrt n}, -\frac1{\sqrt n},\dots, -\frac1{\sqrt n})$.
For this $v$, Cram\'er's theorem shows that $\ln \Pr(\inr{\tilde A}{vv^T} \le
-t n) \ge -\frac{t^{2} n^2}{2p(1-p)}(1 + o(1))$, and the rest of the
proof proceeds as before.



For the claim about the eigenvalue, we use Lemma~\ref{lem:l3-l2}: fix $\eta > 0$ and $K > 0$
and consider the event $\Omega$ on which
\[
    \sum_{\lambda_i(\tilde A) \le -\sqrt {Kn}} \lambda_i^k(\tilde A) \le -t^kn^k
\]
but $\lambda_n^k(\tilde A) \ge -\frac{1}{1+\eta} t^k n^k$. According
to Lemma~\ref{lem:l3-l2}, on this event we have
\[
    \sum_{\lambda_i(\tilde A) \le -\sqrt {Kn}} \lambda_i^2(\tilde A) \ge (1+\eta)^{1/k} t^{2} n^2.
\]
By Corollary~\ref{cor:upper-bound-threshold-conditioned} (for a sufficiently slowly growing $K = \omega(1)$),
\[
    \ln\Pr(\Omega) \le -\frac{t^{2}n^2(1+\eta)^{1/k} \ln \frac{p}{1-p}}{2(2p-1)} + o(t^{2} n^2),
\]
which, for sufficiently large $\eta$ (depending on $p$) implies that
\[
    \ln\Pr(\Omega) \le -(1 + \Omega(1)) \frac{t^{2}n^2}{2p(1-p)}.
\]
It follows from the lower bound in Theorem~\ref{thm:small-p-cycle-count}
that $\Pr(\Omega \mid \tau_k \le p^k - t^k) \to 0$. Together with Lemma~\ref{lem:bulk-contribution}
-- which shows that eigenvalues larger than $-\sqrt{Kn}$ are unlikely to contribute --
this implies that $\lambda_n^k \le -\frac{1}{1+\eta} t^kn^k$ with high probability given $\tau_k \le p^k - t^k$.
(The main difference here compared to the proof of Theorem~\ref{thm:large-p-cycle-count} is
that because we lack matching upper and lower bounds on the log-probabilities, we cannot take $\eta \approx 0$.)

\section{The conditional structure}
\label{sec:conditional}

In our upper bounds on eigenvalue deviation probabilities, we identified a candidate worst-case eigenvector:
specifically, one that takes a certain non-zero value on $\Theta(t n)$ coordinates and zero elsewhere.
In order to identify the conditional structure of this graph, we need to show that this candidate
eigenvector is essentially the only one: every candidate eigenvector that has a comparable deviation
probability is close to the one we identified.

The first step is to characterize the values that give the worst-case result in our Hoeffding-type bounds.
For the rest of this section, fix $p$ and let $\Lambda(u) = \ln (p e^{u(1-p)} + (1-p) e^{-u p})$
be the cumulant-generating function of a centered, $q$-biased Bernoulli variable.

\begin{lemma}\label{lem:robust-Lambda-maximizer}
    The function $\frac{\Lambda(u)}{u^2}$ has a unique maximizer $u_*$, and there is a constant $c = c(p) > 0$
    such that for every $u \in \R$,
    \[
        \frac{\Lambda(u)}{u^2} \le \frac{\Lambda(u_*)}{(u_*)^2} - c \min\{1, (u - u_*)^2\}.
    \]
\end{lemma}

\begin{proof}
    Let $F(u) = \frac{\Lambda(u)}{u^2}$
    (which is continuously defined and differentiable at zero by taking limits).
    Note that $\Lambda(u)$ is asymptotic to $u(1-p)$ as $u \to \infty$ and asymptotic to $-up$ as $u \to -\infty$.
    In particular, $F(u) \to 0$ as $u \to \pm \infty$, and since $F$ is continuous on $\R$ with $F(0) = 0$
    we see that it achieves a maximum at (possibly more than one) $u \in R$.
    Let $\ell^* = \sup_u \frac{\Lambda(u)}{u^2}$, and suppose that $u_*$ achieves the maximum. Then $\Lambda$
    and $u \mapsto \ell^* u^2$ have the same tangent at $u_*$. Since $\Lambda$ is convex, the function $x \mapsto 2 u_* x - \ell^* (u_*)^2$ touches $\Lambda$ from below at $u_*$, and it follows that $\Lambda^*(2 \ell^* u_*) = \ell^* (u_*)^2$.
    Or in other words, $\Lambda^*(s_*) = \frac{(s_*)^2}{4\ell^*}$ for $s_* = 2 \ell^* u_*$. Now recall from
    Lemma~\ref{lem:order-of-optimization} that $\frac{1}{4\ell^*} = \inf_y \frac{\Lambda^*(y)}{y^2}$.
    It follows that for every $u_*$ at which $\Lambda(u)/u^2$ achieves its maximum, there is a $s_*$
    at which $\Lambda^*(y)/y^2$ achieves its minimum. By Lemma~\ref{lem:bernoulli-case}, $\Lambda^*(y)/y^2$
    has a unique minimizer and it follows that $\Lambda(u)/u^2$ has a unique maximizer.

    To see that $\Lambda(u)/u^2$ is locally quadratic near $u_*$, note that $\Lambda''(u_*) (\Lambda^*)''(s_*) = 1$.
    By Corollary~\ref{cor:bernoulli-case}, $(\Lambda^*)''(s_*) > \frac{1}{2\ell^*}$ and it follows that $\Lambda''(u_*) < 2\ell^*$ and so $F(u)$ is locally quadratic near $u_*$. And since $F > 0$ at its unique maximizer and $F(u) \to 0$ at
    $\pm \infty$, it follows that $F(u) \le F(u_*) - c \min\{1,  (u - u_*)^2\}$ for some $c > 0$ and all $u \in \R$.
\end{proof}

With this extra information on the maximizer of $\Lambda(u)/u^2$, we revisit the Hoeffding-type argument
of Proposition~\ref{prop:hoeffding-conclusion}: in order for a matrix $M$ to get close to the upper
bound of Proposition~\ref{prop:hoeffding-conclusion}, most of the contribution to $\|M\|_F$ must come
from entries that are close to the ``ideal value''.
Recall that $s_* = 2p - 1$ is the unique minimizer of $\frac{\Lambda^*(s)}{s^2}$, where $\Lambda^*$ is the convex conjugate
of $\Lambda$. The matrix $M$ that we constructed in Proposition~\ref{prop:hoeffding-conclusion}
had all of its entries being either zero or $s_*/t$; the next result shows that this is essentially necessary.

\begin{proposition}\label{prop:hoeffding-conclusion-tighter}
    With $\xi$ the centered, $p$-biased Bernoulli variable as above, let $A$ be the symmetric
    random matrix with zero diagonal, and with upper-diagonal elements distributed independently
    according to $\xi$. For any $\|M\|_F \le 1$ and $t > 0$,
    \[
        \Pr(\inr{A}{M} \ge t) \le
        \exp\left(
            - \frac{t^2 L}{2} - \Omega\Big(t^2 \sum_i \min\{a_i^2, (a_i - s_*/t)^2\}\Big)
        \right),
    \]
    where $a_i$ are the upper-diagonal entries of $M$.
\end{proposition}

\begin{proof}
    As in Proposition~\ref{prop:hoeffding-conclusion},
    for any $s \in \R$ we have
    \[
        \Pr(\inr{A}{M} \ge t) \le \exp\left(\sum_i \Lambda(s a_i) - st/2\right).
    \]
    By Lemma~\ref{lem:robust-Lambda-maximizer},
    \[
        \sum_i \Lambda(s a_i) = \sum_i \frac{\Lambda (s a_i)}{(s a_i)^2} (s a_i)^2
        \le s^2 \ell^* \sum_i a_i^2 - c s^2 \sum_i a_i^2 \min\{1, (a_i s - u_*)^2\},
    \]
    where $\ell^* = \sup_u \frac{\Lambda(u)}{u^2}$ and $u_* = \frac{s_*}{2 \ell^*} = 2 s_* L$ is the unique maximizer. (Recalling that $s_*$ is the unique minimizer of $\frac{\Lambda^*(s)}{s}$.)
    Since $\sum_i a_i^2 \le \frac 12$,
    choosing $s = t/(2\ell^*) = 2L t$ gives
    \[
        \Pr(\inr{A}{M} \ge t) \le \exp\left(-\frac{t^2 L }{2} - \tilde c t^2 \sum_i a_i^2 \min\{1, (t a_i - s_*)^2\}\right).
    \]
    To simplify the last term, note that because $s_*$ is fixed, $x^2 \min\{1,
    (tx - s_*)^2\} = \min\{x^2, x^2 (tx - s_*)^2\} \ge \Omega(\min\{x^2, (x/t -
    s_*)^2\})$, because if $(tx - s_*)^2 \le 1$ then $x^2 = \Theta(1/t^2)$ and so
    $x^2 (tx - s_*)^2 = \Theta((x - s_*/t)^2)$.
\end{proof}

\begin{corollary}\label{cor:ideal-values}
    Let $A$ be the adjacency matrix of a $\calG(m, n)$ random graph with $m = p \binom n2$ and $p \ge \frac 12$.
    For $c_*$ defined as in~\eqref{eq:c*-def} and any $n^{-c_*} \ll t^k \ll 1$, conditioned on $\tau_k(A) \le p^k - t^k$ the following holds with high probability:
    $\tilde A := A - \E A$ has a unique (up to sign) unit eigenvector $v$ with eigenvalue $\lambda_n(\tilde A)$,
    and it satisfies
    \[
        \sum_{i,j} \min\{v_i^2 v_j^2, (v_i v_j - s_* / (t n))^2\} = o(1).
    \]
\end{corollary}

\begin{proof}
    Uniqueness of the eigenvector follows from Theorem~\ref{thm:large-p-cycle-count}, which implies
    that the eigenvalue $\lambda_n(\tilde A) = -t n(1 - o(1))$ has multiplicity one. Also, Theorem~\ref{thm:large-p-cycle-count}
    implies that
    \[
        \Pr(\tau_k(A) \le p^k - t^k) = \exp\left(-\frac{t^{2} n^2 L}{2}(1 + o(1))\right),
    \]
    and in order to show the claim it suffices to show that the probability of having a unit eigenvector
    $v$ with eigenvalue $-t n(1 - o(1))$ and
    \begin{equation}\label{eq:v-condition}
        \sum_{i,j} \min\{v_i^2 v_j^2, (v_i v_j - s_* / (t n))^2\} \ge \epsilon > 0
    \end{equation}
    is $o(\Pr(\tau_k(A) \le p^k - t^k))$. For $\epsilon > 0$, let $V_\epsilon$ be the set of unit vectors
    $v$ satisfying~\eqref{eq:v-condition}. First, note that for any fixed $v \in V_\epsilon$,
    Proposition~\ref{prop:hoeffding-conclusion-tighter} implies that
    \[
        \Pr(\inr{\tilde A}{v v^T} \le -t n(1-\delta)) \le
        \exp\left(
            - \frac{t^{2} n^2 L}{2}(1 - O(\delta) + \Omega(\epsilon))
        \right).
    \]
    (Proposition~\ref{prop:hoeffding-conclusion-tighter} was written for matrices with i.i.d.\ entries,
    but we can apply it to $\tilde A$ by the standard trick of writing $\tilde A$ as a matrix with
    i.i.d.\ entries, conditioned on having a certain number of positive entries. The probability of the
    event we're conditioning on is $\Omega(1/n)$, and that extra factor of $n$ can be absorbed in
    the $\exp(-\Omega(\epsilon) t^{2} n)$ term.)

    Now let $\calM_{1,\delta} = \{v v^T: v \in V_\epsilon\}$, and by Lemma~\ref{lem:net-size}
    there is a $\delta$-net $\calN$ for $\calM_{1,\epsilon}$ of size at most $(C/\delta)^{Cn}$
    (because we can start with an $(\delta/2)$-net of $\calM_1$ and then project each element of that net
    onto $\calM_{1,\epsilon}$, which gives an $\delta$-net of $\calM_{1,\epsilon}$).
    By Lemma~\ref{lem:net-approximation} and a union bound, for any fixed $\delta > 0$ we have
    \[
        \Pr\left(\inf_{M \in \calM_{1,\epsilon}} \inr{\tilde A}{M}  \le -t n(1-\delta)(1-2\delta)\right)
        \le 
        \exp\left(
            - \frac{t^{2} n^2 L}{2}(1 - O(\delta) + \Omega(\epsilon))
        \right),
    \]
    because the $(C/\delta)^{Cn}$ term coming from the union bound can be
    absorbed in the $\exp(O(t^{2} n^2 \delta))$ term. If $\delta$ is sufficiently small compared to $\epsilon$,
    the probability bound above is asymptotically smaller than $\Pr(\tau_k(A) \le p^k - t^k)$.
    Therefore, for every $\epsilon > 0$,
    conditioned on $\tau_k(A) \le p^k - t^k$, with high probability the eigenvector of $\tilde A$
    with eigenvalue $\lambda_n(\tilde A)$ does not belong to $V_\epsilon$.
\end{proof}

We interpret Corollary~\ref{cor:ideal-values} as saying that for the most-negative eigenvector of $\tilde A$,
most of the $\ell^2$ ``mass'' of $v_i v_j$ is concentrated near $s_* t^{-1} n^{-1}$. Our next task is to
show that (after possibly changing the sign of $v$) $v_i$ essentially takes two values: $0$ and $s_*^{1/2} t^{-1/2} n^{-1/2}$.
For notational convenience, we will adopt a different normalization: consider a vector $w$ with the property that
\be
\label{eq:w-condition}
\sum_{i,j} \min\{w_i^2 w_j^2, (w_i w_j - 1)^2\} \le \epsilon |w|^4,
\ee
and we will show that (after possibly changing the sign of $w$) $w$ is close to taking values zero and 1.
To break the sign symmetry, we will assume without loss of generality that
\be
\label{eq:w-condition-sign}
\sum_{w_i < 0} w_i^2 < \frac 12 |w|^2.
\ee

\begin{lemma}\label{lem:w-close}
    If~\eqref{eq:w-condition} and~\eqref{eq:w-condition-sign} hold then
    \begin{align}
        \sum_{w_i < 0} w_i^2 &\le 2\epsilon |w|^2 \label{eq:w-ineq-negative}, \\
        \sum_{w_i \ge 1} (w_i - 1)^2 &\le \sqrt \epsilon |w|^2 \label{eq:w-ineq-large}, \\
        \sum_{\frac 12 \le w_i \le 1} (w_i - 1)^2 &\le 3 \sqrt \epsilon |w|^2 \label{eq:w-ineq-medium-large}, \text{ and }\\
        \sum_{0 \le w_i \le \frac 12} w_i^2 &\le \sqrt \epsilon |w|^2. \label{eq:w-ineq-small}
    \end{align}

    In particular, if we define $\tilde w$ by setting $\tilde w_i \in \{0, 1\}$, whichever is closer to $w_i$, then
    $|w - \tilde w|^2 \le 6 \sqrt{\epsilon} |w|^2$.
\end{lemma}

\begin{proof}
    If $w_i < 0$ and $w_j > 0$ then $w_i w_j < 0$ and so $w_i^2 w_j^2  \le (w_i w_j - 1)^2$. Hence,
    \[
        \sum_{w_i < 0} \sum_{w_j > 0} w_i^2 w_j^2 \le \sum_{i,j} \min\{w_i^2 w_j^2, (w_i w_j - 1)^2\} \le \epsilon |w|^4.
    \]
    If we define $\gamma$ by $\sum_{w_i < 0} w_i^2 = \gamma |w|^2$ then $\sum_{w_j < 0} w_j^2 = (1 - \gamma)|w|^2$ and so
    the equation above implies that $\gamma(1-\gamma) \le \epsilon$. By~\eqref{eq:w-condition-sign}, $1 - \gamma \ge \frac 12$
    and so $\gamma \le 2\epsilon$. This proves~\eqref{eq:w-ineq-negative}.

    To prove~\eqref{eq:w-ineq-large}, define $\gamma$ by $\sum_{w_i \ge 1} (w_i - 1)^2 = \gamma |w|^2$. 
    Now, $w_i, w_j \ge 1$ implies that $(w_i - 1)^2 (w_j - 1)^2 \le (w_i w_j - 1)^2 \le w_i^2 w_j^2$
    It follows from~\eqref{eq:w-condition} that
    \[
        \gamma^2 |w|^4
        = \sum_{w_i,w_j \ge 1} (w_i - 1)^2 (w_j - 1)^2
        = \sum_{w_i,w_j \ge 1} (w_i w_j - 1)^2
        \le \sum_{i,j} \min\{w_i^2 w_j^2, (w_i w_j - 1)^2\} \le \epsilon |w|^4,
    \]
    and it follows that $\gamma \le \sqrt \epsilon$.

    To prove~\eqref{eq:w-ineq-medium-large}, define $\gamma$ by
    $\sum_{\frac 12 \le w_i \le 1} (w_i - 1)^2 = \gamma |w|^2$. If $\frac 12 \le w_i, w_j \le 1$ then
    $w_i + w_j \ge 2 w_i w_j$ and so
    $1 - w_i w_j \ge (1 - w_i) (1 - w_j)$. Therefore,
    \[
        \gamma^2 |w|^4 
        = \sum_{\frac 12 \le w_i, w_j \le 1} (w_i - 1)^2 (w_j - 1)^2
        \le \sum_{\frac 12 \le w_i, w_j \le 1} (w_i w_j - 1)^2.
    \]
    Now, if $w_i w_j \ge \frac 14$ then $(w_i w_j - 1)^2 \le 9 \min\{w_i^2 w_j^2, (w_i w_j - 1)^2\}$,
    and so~\eqref{eq:w-condition} implies that
    \[
        \gamma^2 |w|^4  \le 9 \epsilon |w|^4.
    \]

    To prove~\eqref{eq:w-ineq-small},
    define $\gamma$ by $\sum_{0 \le w_i \le \frac 12} w_i^2 = \gamma |w|^2$. 
    If $0 \le w_i,w_j \le \frac 12$ then $w_i^2 w_j^2 = \min\{w_i^2 w_j^2, (w_i w_j - 1)^2\}$, and so
    \[
        \gamma^2 |w|^4
        = \sum_{w_i,w_j \ge 1} w_i^2 w_j^2
        \le \sum_{i,j} \min\{w_i^2 w_j^2, (w_i w_j - 1)^2\} \le \epsilon |w|^4,
    \]
    and it follows that $\gamma \le \sqrt \epsilon$.
\end{proof}

We finally come to the proof of Theorem~\ref{thm:conditional-structure}: let $\tilde A$ be the centered adjacency matrix of a $\calG(n,m)$
random graph, and let $v$ be a unit eigenvector with minimal eigenvalue; recall from Corollary~\ref{cor:smallest-eigenvector}
that with high probability $v$ is unique, and that $w := n^{1/2} t^{1/2} s_*^{-1/2} v$ satisfies the condition~\eqref{eq:w-condition}
for some $\epsilon = o(1)$;
for the rest of the proof, we will be working on this event.
Without loss of generality (changing the sign if necessary) $w$ also satisfies~\eqref{eq:w-condition-sign},
and so Lemma~\ref{lem:w-close} implies that there is a vector $\tilde v$ with $\tilde v_i \in \{0, n^{-1/2} t^{-1/2} s_*^{1/2}\}$ for
all $i$, and $|\tilde v - v| = o(1)$. By Theorem~\ref{thm:large-p-cycle-count} and Corollary~\ref{cor:conditional-largest-eigenvalue},
it follows that on this event,
\be\label{eq:event-of-small-innerp}
    \tilde A = -t n \tilde v \tilde v^T + R, \text{ where $\|R\|_{\text{op}} = o(t n)$}.
\ee
Let $U = \{i: \tilde v_i \ne 0\}$.
Since $|\tilde v_i|^2 = 1 + o(1)$
and $v_i^2 \in \{0, n^{-1} t^{-1} s_*\}$, we must have $|U| = (1 + o(1)) n t s_*^{-1}$. 
Now let $V_1$ and $V_2$ be any sets of vertices. Let $1_U$ denote the vector having $(1_U)_i = 1$ for $i \in U$,
and $(1_U)_i = 0$ otherwise; and similarly for $1_{V_1}$ and $1_{V_2}$. Then $\frac 14 \inr{A}{(1_{V_1} + 1_{V_2})^{\otimes 2} - (1_{V_1} - 1_{V_2})^{\otimes 2}}$
counts the number of edges between ${V_1}$ and ${V_2}$. Recalling that $A = \tilde A + p \1 - p I$,
on the event that~\eqref{eq:event-of-small-innerp} holds,
the number of edges between ${V_1}$ and ${V_2}$ is
\begin{equation}\label{eq:cross-edges}
    p |{V_1}| |{V_2}| - (1 + o(1)) \frac 14 (t n \inr{\tilde v}{1_{V_1} + 1_{V_2}}^2 - \inr{\tilde v}{1_{V_1} - 1_{V_2}}^2) + o(t n |1_{V_1} + 1_{V_2}|^2).
\end{equation}
Recall that $\tilde v = n^{-1/2} t^{-1/2} s_*^{1/2} 1_U$. Therefore, if ${V_1}, {V_2} \subset U$ then
$\inr{\tilde v}{1_{V_1}} = n^{-1/2} t^{-1/2} s_*^{1/2} |V_1|$ and similarly for ${V_2}$. Hence,
the number of edges between ${V_1}$ and ${V_2}$ is
\[
    p |{V_1}| |{V_2}| - (1 + o(1)) \frac {s_*}4 ((|{V_1}| + |{V_2}|)^2 - (|{V_1}| - |{V_2}|)^2) + o(|{V_1} \cup {V_2}|^2) =  (1-p) |{V_1}| |{V_2}| + o(|U|^2).
\]
When ${V_1} \subset U^c$, we have $\inr{\tilde v}{1_{V_1}} = 0$ and so~\eqref{eq:cross-edges}
implies that there are $p |{V_1}| |{V_2}| + o(t n (|{V_1}| + |{V_2}|))$ edges between ${V_1}$ and ${V_2}$.
This completes the proof of Theorem~\ref{thm:conditional-structure}
in the case that either $V_1, V_2 \subset U$ or $V_1 \subset U^c$. To obtain the general case, we simply split $V_i$
into $V_i \cap U$ and $V_i \cap U^c$.


\end{document}